\documentclass[11pt]{amsart}
\usepackage[active]{srcltx}
\usepackage{amssymb}
\usepackage{amsmath}
\usepackage{algorithm}
\usepackage{algpseudocode}
\usepackage{enumitem}
\setlist[enumerate]{label = $(\alph*)$, leftmargin = *}
\usepackage{amsthm}
\usepackage{imakeidx}
\usepackage{hyperref}
\usepackage[mathcal]{euscript}
\usepackage{tikz}
\usepackage{xifthen}
\newtheorem{thm}{Theorem}[section]
\newtheorem{cor}[thm]{Corollary}
\newtheorem{lem}[thm]{Lemma}
\newtheorem{deff}[thm]{Definition}
\newtheorem{prop}[thm]{Proposition}
\newtheorem{notacao}[thm]{Notation}
\theoremstyle{remark}

\newcommand{\DRH}{{\sf DRH}}
\newcommand{\DA}{{\sf DA}}
\newcommand{\DS}{{\sf DS}}
\newcommand{\h}{{\sf H}}
\newcommand{\W}{{\sf W}}
\newcommand{\V}{{\sf V}}
\newcommand{\G}{{\sf G}}
\newcommand{\R}{{\sf R}}
\newcommand{\Sl}{{{\sf Sl}}}
\newcommand{\DO}{{{\sf DO}}}

\newcommand{\pseudo}[2]{\overline{\Omega}_{#1} {\sf{#2}}}
\newcommand{\pseudok}[2]{\Omega^{\kappa}_{#1} {\sf{#2}}}

\newcommand{\Aa}{\mathbb{A}}
\newcommand{\zz}{\mathbb{Z}}
\newcommand{\nn}{\mathbb{N}}

\newcommand{\lbf}[2][]{{\sf{lbf}}\ifthenelse{\isempty{#1}}{}{_{#1}}(#2)}
\newcommand{\cum}[1]{\vec{c}(#1)}
\newcommand{\card}[1]{\left\lvert#1\right\rvert}
\newcommand{\leng}[1]{\left\lceil#1\right\rceil}
\newcommand{\just}[2]{\stackrel{#2}{#1}}
\newcommand{\bpi}{\overline{\pi}}
\newcommand{\FG}[1]{{\sf FG}_{#1}}
\newcommand{\cf}{{\sf cf}}
\newcommand{\picf}{\pi_{\sf cf}}
\newcommand{\piir}[1]{\pi_{\sf{irr}}(#1)}
\newcommand{\dyck}[1]{{\sf{Dyck}}(#1)}
\newcommand{\word}[1]{{\sf word}(#1)}
\newcommand{\om}[1]{{\sf om}(#1)}
\newcommand{\reg}[1]{{\sf{reg}}(#1)}
\newcommand{\regi}[1]{{\sf{r.ind}(#1)}}
\newcommand{\id}[1]{{\sf{id}}(#1)}
\newcommand{\irr}[1]{{\sf{irr}}(#1)}
\newcommand{\norm}[1]{\left\lVert#1\right\rVert}

\newcommand{\Req}{\mathrel{\mathcal R}}
\newcommand{\Deq}{\mathrel{\mathcal D}}
\newcommand{\Heq}{\mathrel{\mathcal H}}

\newcommand{\tq}{{{\sf q}}}
\newcommand{\tv}{{{\sf v}}}
\newcommand{\tp}{{{\sf p}}}

\newcommand{\Tt}{{{\sf t}}}
\newcommand{\w}{\overline{w}}

\newcommand{\vx}{\vec{x}}
\newcommand{\vy}{\vec{y}}
\newcommand{\vq}{\vec{q}}

\newcommand{\iG}{\mathcal{G}}
\newcommand{\iL}{\mathcal{L}}
\newcommand{\iA}{\mathcal{A}}
\newcommand{\iT}{\mathcal{T}}
\newcommand{\iF}{\mathcal{F}}
\begin{document}
\title{The $\kappa$-word problem over $\DRH$}
\author{C\'elia Borlido}
\address{Centro de Matem\'atica e Departamento de Matem\'atica, Faculdade de Ci\^encias,
Universidade do Porto, Rua do Campo Alegre, 687, 4169-007 Porto,
Portugal}
\email{cborlido@fc.up.pt}
\begin{abstract}
  Let $\h$ be a pseudovariety of groups in which the $\kappa$-word problem is
  decidable. Here, $\kappa$ denotes the
  canonical implicit signature, which consists of the multiplication and
  the $(\omega-1)$-power.
  We prove that the $\kappa$-word problem is also decidable over
  $\DRH$, the pseudovariety of all finite semigroups whose regular
  $\Req$-classes lie in $\h$.
  Further, we present a canonical form for elements in the free
  $\kappa$-semigroup over $\DRH$, based on the knowledge of a
  canonical form for elements in the free $\kappa$-semigroup
  over~$\h$.
  This extends work of Almeida and Zeitoun on the pseudovariety of
  all finite $\Req$-trivial semigroups.
\end{abstract}
\maketitle
\section{Introduction}
Decidability of word problems has driven researchers'
attention for many years.
Generally speaking, it consists in finding an algorithm (or disprove
its existence) to test whether two
representations of elements of a given structure define
the same element.
On the other hand, pseudovarieties play an essential role since
Eilenberg's correspondence was formulated in 1976 \cite[Chapter VII,
Theorem 3.4s]{MR0530383}.
He showed that pseudovarieties of finite semigroups
are in a bijective correspondence with
varieties of rational languages.
In turn, the study of the latter is strongly motivated by its
application in Computer Science, namely, through the study of Automata
Theory.
For that reason, deciding whether two given expressions have the same
value over all semigroups in a certain pseudovariety seems to be a
relevant question.
Besides these motivations, solving the word problem for $\sigma$-words
(with $\sigma$ an implicit signature) over a pseudovariety $\V$ appears as
an intermediate step 
to prove a stronger property named \emph{tameness} \cite{MR1750489}.
That property on $\V$ has been used to prove decidability of the
membership problem
for pseudovarieties obtained from $\V$ through the application of
join, (two-sided) semidirect product, and Mal'cev product with other
pseudovarieties.

Almeida and Zeitoun \cite{palavra} solved the $\kappa$-word problem over the
pseudovariety~$\R$ of all finite semigroups
whose regular $\Req$-classes are trivial.
Their methods have been extended by Moura~\cite{MR2883026} to
the pseudovariety $\DA$, consisting of all finite semigroups whose regular
$\Deq$-classes are aperiodic subsemigroups.
In this paper, we solve the same problem for some of the pseudovarieties of
the form $\DRH$, containing all finite semigroups whose regular
$\Req$-classes lie in a pseudovariety of groups $\h$.
The only condition we impose on $\h$ is quite reasonable: we require
that it has a decidable $\kappa$-word problem.
Further, combining Moura's work with our own,
it is expected that the same approach may be extended to $\DO \cap
\overline\h$, that is, the pseudovariety of all finite semigroups
whose regular $\Deq$-classes are orthodox semigroups and whose
subgroups lie in~$\h$. The pseudovariety~$\DO \cap \overline \h$ may
be considered as a non aperiodic version of~$\DA$ in the same way
that~$\DRH$ may be seen as a non aperiodic version of~$\R$.
All of these pseudovarieties are specializations of $\DS$, the
pseudovariety of all finite semigroups whose regular $\Deq$-classes
are subsemigroups, which is known to be of particular interest in the
Theory of Formal Languages (see,
for instance,~\cite{MR0444824}).
Hence, their investigation may lead to a better understanding of
$\DS$.

This paper is organized as follows.
Section \ref{section2} of preliminaries is divided into four
subsections:
in the first we set up the general notation; we recall some aspects
related with theory of profinite semigroups in the second; we describe
the $\kappa$-word problem in the third; and we
reserve the fourth to the statement of some general facts on the structure of
the free pro-$\DRH$ semigroup.
In Section \ref{section3} we introduce $\DRH$-automata, which are a
generalization of $\R$-automata defined in \cite{palavra}.
We devote Section \ref{section4} to the presentation of a canonical
form for $\kappa$-words over $\DRH$ assuming the knowledge of a
canonical form for $\kappa$-words over $\h$.
Section~\ref{section5} is rather technical and serves the purpose of
preparing Section~\ref{section6}, in which we describe an algorithm to
solve the $\kappa$-word problem over $\DRH$.
Finally, in Section \ref{section7} we apply our results to the
particular case of the pseudovariety ${\sf DRG}$.
\section{Preliminaries}\label{section2}

We assume the reader is familiar with pseudovarieties, (pro)finite
semigroups, and the basic topology.
For further reading we refer to \cite{livro,profinite,topologia}. 
Some knowledge of automata theory may be useful, although no use
of deep results is made.
For this topic, we refer to~\cite{MR2567276}.
A study of pseudovarieties of the form $\DRH$ may be
found in~\cite{drh}.

\subsection{Notation}

Given a semigroup $S$, we let $S^I$ represent the monoid obtained by
adjoining an identity to~$S$ (even if $S$ is already a monoid).
If $s_1, \ldots, s_n$ are elements in $S$, then $\prod_{i =
  1}^n s_i$ denotes the product $s_1\cdots s_n$.
An infinite sequence $(s_i)_{i \ge 1} \subseteq S$ defines the \emph{infinite
  product} $\left(\prod_{i = 1}^n s_i\right)_{n \ge 1}$.

The free semigroup (respectively, monoid) on a (possibly infinite) set
$C$ is denoted~$C^+$ (respectively, ~$C^*$).
Elements in $C^*$ are called \emph{words}.
The \emph{empty word} of $C^*$ is the identity element $\varepsilon$.
The \emph{length} of a word $u \in C^*$ is $\card u = 0$ if $u =
\varepsilon$, and $\card u = n$ if $u = c_1\cdots c_n$, for certain
$c_1, \ldots, c_n \in C$.
The free group on $C$ is denoted~$\FG C$, and we denote by $C^{-1}$ the
set $\{c^{-1}\colon c \in C\}$ disjoint from~$C$, where~$c^{-1}$
  represents the inverse of~$c$ in~$\FG C$.

We say that a finite set of symbols is an \emph{alphabet}.
Generic alphabets are denoted~$A$, while $\Sigma = \{0,1\}$ is a
fixed two-element alphabet.

Let $\iA = \langle V, \to, \tq, F \rangle$ be a \emph{deterministic
automaton} over an alphabet $A$ (where $V$ is the set of \emph{states},
$\to$ is the \emph{transition 
function}, and $\{\tq\}$ and $F$ are the sets of \emph{initial} and \emph{terminal
states}, respectively).
We write transitions in $\iA$ as $\tv \xrightarrow{a}\tv.a$, for $\tv
\in V$ and $a \in A^*$.
Given a state $\tv \in V$, we denote by $\iA_\tv$ the sub-automaton of
$\iA$ \emph{rooted at} $\tv$, that is, the (deterministic) automaton $\langle
\tv.A^*, \to|_{\tv.A^*}, \tv, F\cap (\tv.A^*)\rangle$.

The symbols $\Req$, $\Heq$, and $\Deq$ denote some of
\emph{Green's relations}.
We reserve the letter~$\h$ to denote an arbitrary pseudovariety of groups, and $\DRH$
stands for the pseudovariety of all finite semigroups whose regular
$\Req$-classes belong to $\h$.
Other pseudovarieties playing a role in this work are
$\S$, the pseudovariety of all finite semigroups;
$\G$, the pseudovariety of all finite groups;
$\R$, the pseudovariety of all finite semigroups with trivial
$\Req$-classes;
$\DS$, the pseudovariety of all finite semigroups whose regular
$\Deq$-classes are subsemigroups;
$\DO$, the pseudovariety of all finite semigroups whose regular
$\Deq$-classes are orthodox subsemigroups;
and $\overline \h$, the pseudovariety of all finite semigroups whose
subgroups lie in $\h$.

\subsection{Profinite semigroups}

Let $\V$ be a pseudovariety of semigroups.
We denote the \emph{free $A$-generated pro-$\V$  semigroup} by
$\pseudo AV$.
Elements of $\pseudo AV$ are called \emph{pseudowords over $\V$} (or simply
\emph{pseudowords}, when $\V = \S$).
Let $\iota:A \to \pseudo AV$ be the generating mapping of $\pseudo AV$.
We point out that, unless $\V$ is the trivial pseudovariety,~$\iota$
is injective.
For that reason, we often identify the alphabet $A$ with its image
under $\iota$.
With this assumption, we obtain that the free semigroup $A^+$ is a
subsemigroup of $\pseudo AV$ and thus, it is coherent to say that $I
\in (\pseudo AV)^I$ is the \emph{empty word/pseudoword}.
On the other hand, if $B \subseteq A$, then we have an injective
continuous homomorphism $\pseudo BV \to \pseudo AV$, induced by the
inclusion map $B \to \pseudo AV$.
So, we consider~$\pseudo BV$ as a subsemigroup of~$\pseudo AV$.
In turn, if~$\W$ is a subpseudovariety of~$\V$, then we denote
by~$\rho_{\V,\W}$ the natural projection from~$\pseudo AV$ onto~$\pseudo
AW$. 
We shall write~$\rho_\W$ when~$\V$ is clear from the context.
Whenever the pseudovariety $\Sl$ of all finite semilattices is
contained in $\V$, we denote the
projection~$\rho_\Sl = \rho_{\V, \Sl}$ by $c$ and call it the \emph{content
  function}.

Finally, a \emph{pseudoidentity over $\V$} (or simply
\emph{pseudoidentity}, when $\V = \S$) is a formal equality $u =
v$, with $u,v \in \pseudo AV$.
We say that a pseudoidentity $u = v$ holds in a pseudovariety $\W
\subseteq \V$ if the interpretations of $u$ and $v$ coincide in every
semigroup of $\W$.
If that is the case, then we write $u =_\W v$.

\subsection{The \texorpdfstring{$\kappa$}{$k$}-word problem}

The \emph{canonical implicit signature}, denoted $\kappa$, consists of
two implicit operations: the \emph{multiplication} $\_\cdot\_$, and the
\emph{$(\omega-1)$-power} $\_^{\omega-1}$.
Each of these operations has a natural interpretation over a given
profinite semigroup $S$: the multiplication sends each pair $(s_1, s_2)$
to its product $s_1s_2$, and the $(\omega-1)$-power sends each element
$s$ to the limit $\lim_{n \ge 1} s^{n!-1}$.
We define \emph{$\kappa$-terms} over an alphabet $A$ inductively as follows:
\begin{itemize}
\item the empty word $I$ and each letter $a \in A$ are $\kappa$-terms;
\item if $u$ and $v$ are $\kappa$-terms, then $(u \cdot v)$ and
  $(u^{\omega-1})$ are also $\kappa$-terms.
\end{itemize}
Of course, each $\kappa$-term may naturally be seen as representing an element of
the free $\kappa$-semigroup $\pseudok AS$, and conversely, for each
element of $\pseudok AS$ there is a (usually non-unique) $\kappa$-term
representing it.
We call \emph{$\kappa$-words} the elements of $\pseudok AS$.

Let $\ell$ be an integer. We may generalize the $(\omega-1)$-power by
letting $x^{\omega+\ell} = \lim_{n \ge 1} x^{n!+\ell}$.
Then, for every $q \ge 1$, the expressions $(x^{\omega-1})^q$ and
$x^{\omega-1}x^{q}$ represent $\kappa$-words (by $u^q$ we mean $q$ times the
product of $u$), and the equalities $(x^{\omega-1})^q = x^{\omega-q}$ and
$x^{\omega-1}x^{q+1} = x^{\omega+q-1}$ hold in $\pseudok AS$.
It is usual to consider the extended implicit signature
$\overline\kappa$ that contains the multiplication and all
$(\omega+q)$-powers (for an integer $q$).
We define both $\overline\kappa$-term and $\overline\kappa$-word in
the same fashion as we defined $\kappa$-term and $\kappa$-word,
respectively.
Clearly, $\kappa$-words are $\overline\kappa$-words and conversely,
but a $\overline\kappa$-term may not be a $\kappa$-term.

Saying that the \emph{$\kappa$-word problem} over a pseudovariety $\V$
is decidable amounts to say that there exists an algorithm
determining whether the interpretation of two given $\kappa$-terms
coincides in every semigroup of $\V$, that is, whether they define the
same element of~$\pseudok AV$.
Although our goal is to solve the $\kappa$-word problem over $\DRH$
(under certain reasonable conditions on~$\h$),
it shall be useful to consider $\overline\kappa$-terms instead
of $\kappa$-terms in the intermediate steps.

The implicit signature $\overline\kappa$ enjoys a nice property that we state
here for later reference.

\begin{lem}
  [{\cite[Lemma 2.2]{palavra}}]
  \label{l:6}
  Let $u$ be a $\overline\kappa$-term and let $u = u_\ell au_r$ be a factorization
  of $u$ such that $c(u) = c(u_\ell)\uplus \{a\}$. Then, $u_\ell$ and
  $u_r$ are $\overline\kappa$-terms.
\end{lem}

\subsection{Structure of free pro-$\DRH$ semigroups}

We start with a uniqueness result on factorization of pseudowords.
\begin{prop}
  [{\cite[Proposition 2.1]{palavra}}]
  \label{p:4}
  Let $x,y,z,t \in \pseudo AS$ and $a,b \in A$ be such that $xay =
  zbt$. Suppose that $a \notin c(x)$ and $b \notin c(z)$. If either
  $c(x) = c(z)$ or
  $c(xa) = c(zb)$, then $x = z$, $a = b$, and $y = t$.
\end{prop}

This motivates the definition of \emph{left basic factorization} of a
pseudoword $u \in \pseudo AS$: it is the unique triple $\lbf u = (u_\ell, a,
u_r)$ of $(\pseudo AS)^I \times A \times (\pseudo AS)^I$ such that $u =
u_\ell a u_r$, $a \notin c(u_\ell)$, and $c(u) = c(u_\ell a)$.
The left basic factorization is also well defined over each
pseudovariety~$\DRH$.
\begin{prop}
  [{\cite[Proposition 2.3.1]{drh}}]\label{p:1}
  Every element $u \in \pseudo A{DRH}$ admits a unique factorization
  of the form $u = u_\ell a u_r$ such that $a \notin c(u_\ell)$
  and $c(u_\ell a) = c(u)$.
\end{prop}
Then,  whenever $u \in \pseudo A{DRH}$, we also say that the triple
$\lbf u = (u_\ell, a, u_r)$ described in Proposition~\ref{p:1} is the
left basic factorization of~$u$.

We may iterate the left basic factorization of a pseudoword $u$ (over
$\DRH$) as follows.
Set $u_0' = u$. For $k \ge 0$, if $u_k' \neq I$, then we let
$(u_{k+1}, a_{k+1}, u_{k+1}')$ be the left basic factorization of
$u_k'$.
Since the contents $(c(u_ka_k))_{k\ge 1}$ form a decreasing
sequence for inclusion, either there exists an index $k$ such that
$u_k' = I$ or, for all~$m \ge k$, $c(u_ka_k) = c(u_ma_m)$.
The \emph{cumulative content} of $u$ is $\cum u =
\emptyset$ in the former case, and it is $\cum u
= c(u_ka_k)$ otherwise.
In particular, Proposition~\ref{p:1} yields that the cumulative
content of a pseudoword is completely determined by its projection
onto $\pseudo AR$.
We denote the factor $u_ka_k$ by $\lbf[k]u$, whenever it is defined
and we write $\lbf[\infty] u = (u_1a_1,\ldots, u_ka_k, I, I, \ldots)$
if $u_k' = I$, and $\lbf[\infty] u = (u_ka_k)_{k\ge 1}$ otherwise.
We further define the \emph{irregular} and \emph{regular} parts
of~$u$, respectively denoted $\irr u$ and $\reg u$: if $\cum u =
\emptyset$, then $\irr u = u$ and $\reg u = I$;
if $\cum u = c(u_k')$ and $k$ is minimal for this equality, then
$\irr u = \lbf[1]u \cdots \lbf[k] u$ and $\reg u = u_k'$.
This terminology is explained by the following result.
\begin{prop}
  [{\cite[Corollary 6.1.5]{drh}}] \label{c:8}
  Let $u \in \pseudo A{DRH}$. Then, $u$ is regular if and only if
  $c(u) = \cum u$ (and, hence, $\reg u = u$).
\end{prop}

Suppose that $\cum u \neq \emptyset$.
Since $\pseudo AS$ is a compact monoid, it follows that
the infinite product $(\lbf[1] u \cdots \lbf[k] u)_{k \ge 1}$ has an
accumulation point, and it is not hard to see that any two of its
accumulation points are $\Req$-equivalent.
Furthermore, if all the factors $\lbf[k] u$ have the same content, then
the $\Req$-class in which the accumulation points lie is
regular~\cite[Proposition 2.1.4]{drh}.
On the other hand, the regular $\Req$-classes of $\pseudo A{DRH}$ are
groups.
Hence, in this case, we may define the \emph{idempotent designated} by
the infinite product $(\lbf[1] u \cdots \lbf[k] u)_{k \ge 1}$ to be the
identity of the group to which its accumulation points belong.

Together with Lemma \ref{l:6}, the next result is behind the properties of
$\pseudo A{DRH}$ that we use most often in the sequel.
\begin{prop}
  [{\cite[Proposition 5.1.2]{drh}}]
  \label{p:3}
  Let $\V$ be a pseudovariety such that the inclusions $\h\subseteq \V
  \subseteq \DO \cap \overline \h$ hold. If $e$ is
  an idempotent of $\pseudo AV$ and if $H_e$ is its $\Heq$-class, then
  letting $\psi_e(a) = eae$ for each $a \in c(e)$ defines a unique
  homeomorphism $\psi_e:\pseudo{c(e)}H \to H_e$ whose inverse is the
  restriction of $\rho_{\V, \h}$ to $H_e$.
\end{prop}
The following consequence is not hard to derive.
\begin{cor}
  \label{c:2}
  Let $u$ be a pseudoword and $v,w \in (\pseudo AS)^I$ be such that
  $c(v)\cup c(w) \subseteq \cum u$ and $v =_\h w$. Then, the
  pseudovariety $\DRH$ satisfies $uv = uw$. 
\end{cor}

We proceed with the statement of two known facts about $\DRH$.
Their proofs may be found in~\cite{phd}.

\begin{lem}
  \label{sec:13}
  Let $u, v$ be pseudowords. Then, $\rho_{\DRH} (u)$ and
  $\rho_\DRH(v)$ lie in the same $\Req$-class if and only if the
  pseudovariety $\DRH$ satisfies $\lbf[\infty] u = \lbf[\infty] v$. 
\end{lem}

\begin{lem}
  \label{sec:16}
    Let $u,v \in \pseudo AS$ and $u_0, v_0 \in (\pseudo AS)^I$ be such
  that $c(u_0) \subseteq \cum u$ and $c(v_0) \subseteq \cum
    v$. Then, the pseudovariety $\DRH$ satisfies $uu_0 = vv_0$ if and
  only if it satisfies $u \Req v$ and if, in addition, the pseudovariety
  $\h$ satisfies $uu_0 = vv_0$. In particular, by taking $u_0 = I =
  v_0$, we get that $u =_\DRH v$ if and only if $u \Req v$ modulo $\DRH$ and $ u =_\h v$.
\end{lem}
\section{$\DRH$-automata}\label{section3}
We start by introducing the notion of a $\DRH$-automaton.

\begin{deff}
  An \emph{$A$-labeled $\DRH$-automaton} is a tuple $$\iA = \langle V, \to,
  \tq, F, \lambda_\h, \lambda \rangle,$$ where $\langle V, \to, \tq,
  F\rangle$ is a nonempty deterministic trim automaton over
  $\Sigma$ and  $\lambda_\h: V \to (\pseudo AH)^I$ and
  $\lambda: V \to A \uplus \{\varepsilon\}$ are functions.
  We further require that $\iA$ satisfies the following conditions {\ref{a1}--\ref{a6}}.
  \begin{enumerate}[label = (A.\arabic*), leftmargin = 1.2cm]
  \item\label{a1} the set of final states is $F =
    \lambda^{-1}(\varepsilon)$ and
      $\lambda_\h(F) = \{I\}$;
  \item\label{a2} there is no outgoing transition from
    $F$;
  \item\label{a3} for every $\tv \in V \setminus F$,
    both $\tv.0$ and $\tv.1$ are defined;
  \item\label{a4} for every $\tv \in V \setminus F$, the
    equality $\lambda(\tv.\Sigma^*) = \lambda(\tv.0\Sigma^*) \uplus
    \{\lambda(\tv)\}$ holds.
  \end{enumerate}
  We observe that if conditions \ref{a1}--\ref{a4} hold for $\iA$,
  then the \emph{reduct}
  $\iA_\R = \langle V, \to, \tq, F, \lambda
  \rangle$ is an $A$-labeled $\R$-automaton (see \cite[Definition
  3.11]{palavra}). Since the cumulative content of
  a pseudoword over $\DRH$ depends only on its projection onto $\pseudo
  AR$, and hence, also its regularity, we may use the known
  results for the word problem in $\R$ (namely,
  \cite[Theorem~3.21]{palavra}) as intuition for defining the
  \emph{length}
  $\norm \iA$, the \emph{regularity index} $\regi\iA$ and the
  \emph{cumulative content} $\cum{\iA}$ of a $\DRH$-automaton
  $\iA$ from the knowledge of its reduct~$\iA_\R$. We set:
  \begin{align*}
  \norm \iA   &= \sup \{ k \ge 0 \colon \tq.1^k \text{ is defined}\};
    \\ \regi \iA   &=
         \begin{cases}
           \infty, \quad\text{ if } \norm \iA < \infty;
           \\ \min \{m \ge 0 \colon \forall k \ge m \quad
           \lambda(\tq.1^k\Sigma^*) = \lambda(\tq.1^m\Sigma^*)\},
           \quad\text{ otherwise};
         \end{cases}
    \\ \cum \iA   &=
         \begin{cases}
           \emptyset, \quad\text{ if } \norm \iA < \infty;
           \\ \lambda(\tq.1^{\regi\iA}\Sigma^*), \quad\text{ otherwise}.
         \end{cases}
  \end{align*}
  We are now able to state the further required properties
  for $\iA$:
  \begin{enumerate}[label = (A.\arabic*), leftmargin = 1.2cm]\setcounter{enumi}{4}
  \item\label{a5} if $\tv \in V \setminus F$, then
    $\lambda_\h(\tv) = I$ if and only if $\norm {\iA_{\tv.0}} <
    \infty$;
  \item\label{a6} if $\tv \in V \setminus F$ and
    $\norm{\iA_{\tv.0}} = \infty$, then $\lambda_\h(\tv) \in \pseudo
    {\cum{\iA_{\tv.0}}} H$.
  \end{enumerate}
  We say that $\iA$ is a \emph{$\DRH$-tree} if it is a $\DRH$-automaton such
  that for every $\tv \in V$ there exists a unique $\alpha \in
  \Sigma^*$ such that $\tq.\alpha = \tv$.
\end{deff}
\begin{deff}
  We say that two $\DRH$-automata $\iA_i = \langle V_i, \to_i, \tq_i,
  F_i, \lambda_{i,\h}, \lambda_i\rangle$, $i = 1,2$, are
  \emph{isomorphic}
  if there exists a bijection $f : V_1 \to V_2$ such that
  \begin{itemize}
  \item\label{i1} $f(\tq_1) = \tq_2$;
  \item\label{i2} for every $ \tv \in V_1$ and $\alpha
    \in \Sigma$, $f(\tv)\cdot\alpha = f(\tv \cdot \alpha)$;
  \item\label{i3} for every $\tv \in V_1$, the equalities
    $\lambda_{1,\h}(\tv) = \lambda_{2,\h}(f(\tv))$ and 
    $\lambda_{1}(\tv) = \lambda_2(f(\tv))$ hold.
  \end{itemize}
\end{deff}
Isomorphic $\DRH$-automata are essentially the same, up to the name of
the states. Therefore, we consider $\DRH$-automata only up to
isomorphism.

We denote the trivial $\DRH$-automaton by ${\bf 1}$ and the set of all
$A$-labeled $\DRH$-automata by $\Aa_A$.
\begin{deff}\label{d:1}
  Let $k \ge 0$ and $\iA_i = \langle V_i,\to_i, \tq_i, F_i,
  \lambda_{i,\h}, \lambda_i \rangle$, $i = 1,2$, be two
  $\DRH$-automata. We say that $\iA_1$ and $\iA_2$ are
  \emph{$k$-equivalent} if
  \begin{equation}
    \begin{aligned}
      \forall \alpha \in \Sigma^*, \: \card \alpha \le k
      \implies
      \begin{cases}
        \lambda_1(\tq_1.\alpha) = \lambda_2(\tq_2.\alpha); \\ 
        \lambda_{1,\h}(\tq_1.\alpha) = \lambda_{2,\h}(\tq_2.\alpha).
      \end{cases}
    \end{aligned} \label{eq:10}
  \end{equation}
  If $\iA_1$ and $\iA_2$ are $k$-equivalent for every $k \ge 0$, then
  we say that they are 
  \emph{equivalent}.
  We write $\iA_1 \sim_k \iA_2$ (respectively, $\iA_1 \sim \iA_2$), when $\iA_1$
  and $\iA_2$ are $k$-equivalent (respectively, equivalent).
  We further agree that \eqref{eq:10} means that either both equalities hold
  or both $\tq_1.\alpha$ and $\tq_2.\alpha$ are undefined.
\end{deff}

Observe that equivalent $\DRH$-trees are necessarily
isomorphic.

The following lemma is useful when defining a bijective
correspondence between 
the equivalence classes of $\Aa_A$ and the $\Req$-classes of $\pseudo
A \DRH$.
Although its proof is analogous to the proof of \cite[Lemma
3.16]{palavra}, we include it for the sake of completeness.
\begin{lem}\label{3.16}
  Every $\DRH$-automaton has a unique equivalent $\DRH$-tree.
\end{lem}
\begin{proof}
  Take a $\DRH$-automaton $\iA = \langle V, \to, \tq, F,
  \lambda_{\h}, \lambda \rangle$ and let $\iT = \langle V', \to',
  \tq', F', \lambda_\h', \lambda'\rangle$ be the $\DRH$-tree defined
  as follows. We set
  $V' = \{\alpha \in \Sigma^* \colon \tq.\alpha \text{ is defined}\}$
  and put $\tq' = \varepsilon$. The labels of each state $\alpha \in
  V'$ are given by $\lambda_\h'(\alpha) = \lambda_\h(\tq.\alpha)$ and
  by $\lambda'(\alpha) = \lambda(\tq.\alpha)$. We also take $F' =
  \lambda'^{-1}(\varepsilon)$. Finally, the transitions in $\iT$ are
  given by $\alpha.0 = \alpha0$ and by $\alpha.1 = \alpha1$, whenever
  $\lambda'(\alpha) \neq \varepsilon$.
  It is a routine matter to check that $\iT$ is a
  $\DRH$-tree equivalent to $\iA$.
\end{proof}

Given a $\DRH$-automaton $\iA$, we denote by $\vec{\iA} = \langle \vec
V, \to, \vec \tq, \vec F, \vec \lambda_\h, \vec \lambda \rangle$ the
unique $\DRH$-tree which is equivalent to $\iA$.
Denoting both transition functions of $\iA$ and of $\vec \iA$ by $\to$
is an
abuse of notation justified by the construction made in the proof of
Lemma \ref{3.16}.
Given $0 \le i \le
\norm \iA$, we denote by $\iA_{[i]}$ the $\DRH$-subtree rooted at
$\vec \tq.1^i0$.

\begin{notacao}
  Let $u \in \pseudo A{DRH}$ and $v \in \pseudo AH$ be such that
  $c(v)\subseteq \cum u$. By
  Corollary~\ref{c:2}, the set $u \rho_{\DRH,\h}^{-1}(v)$ is a
  singleton.
  It is convenient to denote by $uv$ the unique element of
  $u\rho_{\DRH,\h}^{-1}(v)$.
  In this case, the notation $\rho_\h
  (uv)$ refers to the element $\rho_\h(uv) = \rho_\h(u) \: v$ of~$\pseudo AH$.
\end{notacao}
\begin{deff} \label{sec:22}
  Let $\iA = \langle V, \to, \tq, F, \lambda_\h, \lambda \rangle$
  be an $A$-labeled $\DRH$-automaton. The \emph{value}
  $\pi(\iA)$ of $\iA$ in $(\pseudo A {DRH})^I$ is inductively
  defined as follows:
  \begin{itemize}
  \item if $\iA = {\bf 1}$, then $\pi(\iA) = I$;
  \item otherwise, we consider two different cases according to
    whether or not $\norm{\iA} < \infty$.
    \begin{itemize}
    \item If  $\norm{\iA} < \infty$, then we set
      \begin{align*}
        \pi(\iA) = \prod_{i = 0}^{\norm{\iA} - 1} \pi(\iA_{[i]})
        \lambda_{\h}(\tq.1^i) \lambda(\tq.1^i).
      \end{align*}
    \item If  $\norm{\iA} = \infty$, then we first define the
      \emph{idempotent associated to $\iA$}, denoted $\id \iA$.
      Noticing that, for $k \ge \regi\iA$, all the elements
      $\pi(\iA_{[k]}) \lambda_{\h}(\tq.1^k) \lambda(\tq.1^k)$ have the
      same content, we let $\id \iA$ be the idempotent designated by
      the infinite product
      \begin{equation}
        (\pi(\iA_{[\regi\iA]}) \lambda_{\h}(\tq.1^{\regi\iA})
        \lambda(\tq.1^{\regi\iA}) \cdots \pi(\iA_{[k]})
        \lambda_{\h}(\tq.1^k) \lambda(\tq.1^k))_{k \ge
          \regi\iA}.\label{eq:61}
      \end{equation}
      Then, we take
      \begin{align*}
        \pi(\iA) = \left(\prod_{i = 0}^{\regi\iA - 1} \pi(\iA_{[i]})
        \lambda_{\h}(\tq.1^i) \lambda(\tq.1^i)\right) \cdot \id\iA.
      \end{align*}
    \end{itemize}
  \end{itemize}
  
  We also define the \emph{value of the irregular part of $\iA$}:
  \begin{align*}
    \piir\iA = \prod_{i = 0}^{\min\{\norm{\iA}, \regi\iA\} -
    1} \pi(\iA_{[i]}) \lambda_{\h}(\tq.1^i) \lambda(\tq.1^i).
  \end{align*}
  If $\norm{\iA} < \infty$, then we set $\id\iA = I$. Using this
  notation, we have the equality
  \begin{equation}
    \pi(\iA) = \piir\iA \cdot \id \iA.\label{eq:32}
  \end{equation}
\end{deff}

The next result is a simple observation that we state for later
reference.
\begin{lem}
  \label{l:2}
  Given a $\DRH$-automaton $\iA = \langle  V, \to, \tq, F, \lambda_\h,
  \lambda \rangle$, the following equalities hold:
  \begin{align*}
    \lbf[i+1]{\pi(\iA)} &= \pi(\iA_{[i]})
    \lambda_\h(\tq.1^i)\lambda(\tq.1^i), \text{ whenever
                          $\lbf[i+1]{\pi(\iA)}$ is defined};
    \\ \irr{\pi(\iA)} &= \piir \iA;
    \\\cum\iA &= \cum{\pi(\iA)}.
  \end{align*}
  In particular, for a certain $u \in \pseudo A{DRH}$, the elements
  $\pi(\iA)$ and $u$ are
  $\Req$-equivalent if and only if $\piir \iA = \irr u$ and
  $\id \iA \Req \reg u$.\qed
\end{lem}

Since the value of a $\DRH$-automaton $\iA$ depends only on the unique
$\DRH$-tree $\vec \iA$ lying in the $\sim$-class of $\iA$, there is a
well defined map $\overline \pi: \Aa_A/{\sim} \to (\pseudo A
{DRH})^I/{\Req}$ which sends a class $\iA/{\sim}$ to the $\Req$-class
of the value of
$\vec \iA$. This map is, in effect, a bijection.
\begin{thm}\label{3.21}
  The map $\overline \pi$
  is bijective.
\end{thm}
\begin{proof}
  To prove that $\bpi$ is injective, we consider two $\DRH$-automata
  $\iA = \langle V, \to, \tq, F, \lambda_{\h}, \lambda
  \rangle$ and $\iA' = \langle V', \to', \tq', F',
  \lambda'_{\h}, \lambda' \rangle$ such that $\pi(\iA) \Req \pi(\iA')$
  and we argue by induction on $\card{c(\pi(\iA))} =
  \card{c(\pi(\iA'))}$.

  If $\card{c(\pi(\iA))}  = 0$, then $\iA = {\bf 1} = \iA'$ and there is nothing
  to prove.
  
  Suppose that $ \card{c(\pi(\iA))} > 0$. We claim that $\iA_{[i]} =
  \iA'_{[i]}$ for all $ 0 \le i \le \norm{\iA} -1$.
  Indeed, by Lemma~\ref{sec:13}, the values $\pi(\iA)$ and $\pi(\iA')$ lie
  in the same $\Req$-class if and only if $\lbf[\infty]{\pi(\iA)} =
  \lbf[\infty]{\pi(\iA')}$.
  But, by Lemma \ref{l:2}, the equalities
  \begin{align*}
    \lbf[i+1]{\pi(\iA)} &= \pi(\iA_{[i]})
                          \lambda_\h(\tq.1^i)\lambda(\tq.1^i)
    \\ \lbf[i+1]{\pi(\iA')} & = \pi(\iA'_{[i]})
  \lambda'_\h(\tq'.1^i)\lambda'(\tq'.1^i)
  \end{align*}
  hold, whenever the first members are defined.
  Hence, we get the following:
  \begin{align}\label{eq:56}
   \norm{\iA}  &= \norm{\iA'},\nonumber
    \\ \pi(\iA_{[i]}) \lambda_{\h}(\tq.1^i)  &= \pi(\iA'_{[i]})
         \lambda'_{\h}(\tq'.1^i), \text{ for } 0 \le i \le \norm{\iA}-1,
    \\ \lambda(\tq.1^i)  &=  \lambda'(\tq'.1^i), \text{ for } 0 \le
         i \le \norm{\iA}-1.\nonumber
  \end{align}
  Since, by \ref{a6}, the inclusions $c(\lambda_{\h}(\tq.1^i)) \subseteq
  \vec{c}( \pi(\iA_{[i]}))$ and $c(\lambda'_{\h}(\tq'.1^i))
  \subseteq \vec{c}( \pi(\iA'_{[i]}))$ hold, we also have $
  \pi(\iA_{[i]}) \Req \pi(\iA'_{[i]})$. By induction hypothesis,
  that implies $\iA_{[i]} = \iA'_{[i]}$ (recall that $\iA_{[i]}$ and
  $\iA'_{[i]}$ are both $\DRH$-trees, and each equivalence class has
  a unique $\DRH$-tree).
  
  To conclude that $\bpi$ is injective, it remains to show that, for
  $0 \le i \le \norm{\iA} - 1$, the labels  $\lambda_{\h}(\tq.1^i)$
  and $\lambda'_{\h}(\tq'.1^i)$ coincide.
  When $\vec{c}(\iA_{[i]}) = \emptyset = \vec{c}(\iA'_{[i]})$,
  Property \ref{a6} guarantees that $\lambda_{\h}(\tq.1^i) = I =
  \lambda'_{\h}(\tq'.1^i)$.
  Otherwise, we have
  \begin{align*}
    \piir {\iA_{[i]}} \id{\iA_{[i]}} \lambda_{\h}(\tq.1^i) & =
    \pi(\iA_{[i]}) \lambda_{\h}(\tq.1^i)
    \just ={\eqref{eq:56}}
  \pi(\iA'_{[i]}) \lambda'_{\h}(\tq'.1^i)  \\ &= \piir{\iA'_{[i]}}
  \id{\iA'_{[i]}} \lambda'_{\h}(\tq'.1^i),
  \end{align*}
  which in turn implies that
  $$\id{\iA_{[i]}} \lambda_{\h}(\tq.1^i) = \id{\iA'_{[i]}}
  \lambda'_{\h}(\tq'.1^i).$$
  Since $\rho_\h(\id{\iA_{[i]}})$ and $\rho_\h(\id{\iA'_{[i]}})$ are both the
  identity of $\pseudo AH$ we obtain the equality $\lambda_{\h}(\tq.1^i) =
  \lambda'_{\h}(\tq'.1^i)$.
  
  Let us prove that $\bpi$ is surjective. We proceed again by
  induction, this time on $\card{c(w)}$, for $w \in (\pseudo A
  {DRH})^I$.
  
  If $c(w)$ is the empty set, then we have $[w]_{\Req} = \{I\} =
  \{\pi({\bf 1})\} = \bpi({\bf 1}/_{\sim})$.
  
  If $w \neq I$, then we let $w = w_0a_0 \cdots w_ka_kw'_k$ be the $k$-th
  iteration of the left basic factorization of $w$ (whenever it is
  defined). For each $0 \le i \le \leng{w} - 1$, we have $c(w_i)
  \subsetneqq c(w)$ and so, by induction hypothesis, there exists a
  $\DRH$-tree $\iA_i = \langle V_i, \to_i, \tq_i, F_i,
  \lambda_{i,\h}, \lambda_i \rangle$ such that $\pi(\iA_i) \Req
  w_i$. In particular, the equality $ \piir {\iA_i} = \irr {w_i}$
  holds and consequently, $\h$ satisfies
  \begin{align}
    \pi(\iA_i) \cdot \reg{w_i} = \piir{\iA_i} \cdot \id
    {\iA_i} \cdot \reg{w_i} = \irr{w_i}\cdot 1 \cdot \reg{w_i}  =
    w_i.\label{eq:18}
  \end{align}
  On the other hand, since
  $c(\reg{w_i}) = \cum{\id{\iA_i}}$,
  we deduce that $\id{\iA_i} \cdot \reg{w_i}$ is $\Req$-equivalent
  to $\id{\iA_i}$. Consequently, the pseudowords $w_i$ and $\pi(\iA_i)
  \cdot \reg{w_i}$ are $\Req$-equivalent as well. This relation
  together with~\eqref{eq:18} imply, by Lemma~\ref{sec:16}, that the
  equality 
  $\pi(\iA_i) \cdot \reg{w_i} = w_i $ holds.
  
  Now, we construct a $\DRH$-tree $\iA = \langle V, \to,
  \tq, F, \lambda_{\h}, \lambda \rangle $ as follows:
  \begin{itemize}
  \item $ V =
    \begin{cases}
      \{ \tv \in V_i \colon i \ge 0\} \uplus \{\tv_i\}_{i \ge 0}, \quad
      \text{ if $\leng{w} = \infty;$}
      \\ \{ \tv \in V_i \colon i = 0, \ldots, \leng w - 1\} \uplus
      \{\tv_i\}_{i = 0}^{\leng{w} - 1}\uplus
      \{\tv_{\varepsilon}\}, \;\text{ if }
      \leng{w} < \infty;
    \end{cases}$
  \item $\tq = \tv_0$;
  \item $ F =\begin{cases}
      \{ \tv \in F_i\colon i \ge 0\}, \quad \text{ if } \leng{w} = \infty;
      \\
      \{ \tv \in F_i \colon i = 0, \ldots, \leng w -1\} \uplus
      \{\tv_{\varepsilon}\}, \quad \text{ if } \leng{w} < \infty;
    \end{cases}
    $
  \item $\lambda_{\h}(\tv_i) = \rho_{\h}(\reg{w_i})$ and
    $\lambda(\tv_i) = a_i$ for $i = 0, \ldots,  \leng w-1$;
  \item $\lambda(\tv_{\varepsilon}) = \varepsilon$, if $\leng{w}$ is
    finite;
  \item $\tv_i.0 = \tq_i $ and $ \tv_i.1 =
    \begin{cases}
      \tv_{i+1}, & \text{ if } i < \leng{w} - 1; \\
      \tv_{\varepsilon}, & \text{ if } i = \leng{w} - 1;
    \end{cases}
    $
  \item transitions and labelings on $V_i$ are given by those of
    $\iA_i$.
  \end{itemize}
  Then it is easy to check that $\iA$ is a $\DRH$-tree and that
  $\overline\pi(\iA/{\sim}) = [w]_{\Req}$.
\end{proof}

Suppose that we are given two $\DRH$-automata $\iA_i = \langle V_i,
\to_i, \tq_i, F_i, \lambda_{i,\h}, \lambda_i \rangle$, $i =
0,1$, a letter $a \in A$ such that  $\lambda(V_1) \subseteq
\lambda(V_0) \uplus \{a\}$ and a pseudoword $u$ such that $c(u)
\subseteq \vec{c}(\iA_0)$. Then, we denote by $(\iA_0, u \mid a,
\iA_1)$ the $\DRH$-automaton $\iA = \langle V, \to, \tq, F,
\lambda_{\h}, \lambda \rangle$, where
\begin{itemize}
\item $V = V_0 \uplus V_1 \uplus \{\tq\}$;
\item $\tq.0 = \tq_0$ and $\tq.1 = \tq_1$;
\item $ F = F_0 \uplus F_1$;
\item $\lambda_{\h} (\tq) = \rho_\h(u)$ and $\lambda(\tq) = a$;
\item all the other transitions and labels are given by those of
  $\iA_0$ and $\iA_1$.
\end{itemize}
Given an element $w$ of $(\pseudo AS)^I$, we denote by $\iT(w)$ the $\DRH$-tree
representing the $\sim$-class $\overline\pi^{\:-1}([\rho_{\DRH}(w)]_{\Req})$.
With a little abuse of notation, when $w \in (\pseudo A{DRH})^I$, we
use $\iT(w)$ to denote the unique
$\DRH$-tree in the $\sim$-class $\overline\pi^{\:-1}([w]_{\Req})$.
Later, we shall see that, for every $\kappa$-word $w$, there
exists a finite $\DRH$-automaton $\iA$ in the $\sim$-class of $\iT(w)$ (Corollary~\ref{c:14}).
\begin{lem}\label{arv}
  Let $w$ be a pseudoword and write $\lbf w = (w_\ell, a, w_r)$. Then,
  we have the equality $\iT(w)
  = (\iT(w_\ell), \reg{w_\ell} \mid a, \iT(w_r))$.
\end{lem}
\begin{proof}
  Write
  \begin{align*}
    \iT ' &= (\iT(w_\ell), \reg{w_\ell} \mid a, \iT(w_r)) = \langle V,
    \to, \tq, F, \lambda_{\h}, \lambda \rangle;
    \\ \iT(w_\ell) &= \langle V_0,
                     \to_0, \tq_0, F_0, \lambda_{0,\h}, \lambda_0 \rangle;
    \\ \iT(w_r) &= \langle V_1,
  \to_1, \tq_1, F_1, \lambda_{1,\h}, \lambda_1 \rangle.
  \end{align*}
  The claim
  amounts to proving that $\pi (\iT') \Req w$ modulo
  $\DRH$. By definition of $\iT'$, we have $\norm {\iT'} <
  \infty$ if and only if $\norm {\iT(w_r)} < \infty$.
  We start by proving that $\pi(\iT')$ and $\pi(\iT(w_\ell))
  \rho_\h(\reg{w_\ell}) a \cdot \pi( \iT(w_r))$ belong to the same $\Req$-class.
  It is worth noticing that, for every $1 \le i \le \norm{\iT'}$, we
  have the following equality:
  \begin{equation}
    \label{eq:33}
    \iT'_{[i]} = \iT_{\tq.1^i0}' = \iT(w_r)_{\tq_1.1^{i-1}0} =
    \iT(w_r)_{[i-1]}.
  \end{equation}
  
  First, assume that $\norm {\iT'} < \infty$.
  Then, we have $\norm {\iT'} = \norm {\iT(w_r)} + 1$.
  Following Definition~\ref{sec:22} and the construction of $\iT'$, we may
  compute
  \begin{align}
    \pi(\iT') & = \prod_{i = 0}^{\norm {\iT(w_r)}} \pi(\iT'_{[i]})
                \lambda_\h(\tq.1^i)\lambda(\tq.1^i)\nonumber
    \\ & = \pi(\iT'_{\tq.0}) \lambda_\h(\tq)\lambda(\tq) \cdot
         \prod_{i = 0}^{\norm {\iT(w_r)}-1}
         \pi(\iT'_{[i+1]})
         \lambda_\h(\tq.1^{i+1})\lambda(\tq.1^{i+1})\nonumber
    \\ &\kern-3pt \just={\eqref{eq:33}} \pi(\iT(w_\ell))
         \rho_\h(\reg{w_\ell}) a \cdot \pi(\iT(w_r)).\label{eq:34}
  \end{align}

  Now, we suppose that $\norm{\iT'} = \infty$. In that case,
  $\regi{\iT'}$ is either $\regi{\iT(w_r)}$ or $\regi{\iT(w_r)} +
  1$ according to whether $\rho_\DRH(w)$ is regular (in which case, it
  is $0$) or not, respectively. Suppose that $\rho_\DRH(w)$ is not
  regular. We compute
  \begin{align*}
    \pi(\iT') & = \prod_{i = 0}^{\regi {\iT(w_r)}} \pi(\iT'_{[i]})
                \lambda_\h(\tq.1^i)\lambda(\tq.1^i) \cdot \id{\iT'}
    \\ &= \pi(\iT'_{\tq.0})
         \lambda_\h(\tq)\lambda(\tq)
         \cdot \left(\prod_{i = 0}^{\regi {\iT(w_r)}-1} \pi(\iT'_{[i+1]})
         \lambda_\h(\tq.1^{i+1})\lambda(\tq.1^{i+1})\right) \cdot
         \id{\iT'}
    \\ & \kern-3pt\just={\eqref{eq:33}}\pi(\iT(w_\ell))
         \rho_\h(\reg{w_\ell})a
         \cdot \piir{\iT(w_r)}
         \cdot \id{\iT'}.
  \end{align*}
  Now, $\id{\iT'}$ is the idempotent designated by the infinite
  product
  $$(\pi(\iT'_{[\regi{\iT'}]})\lambda_\h(\tq.1^{\regi{\iT'}})\lambda(\tq.1^{\regi{\iT'}})\cdots
  \pi(\iT'_{k})\lambda_\h(\tq.1^{k})\lambda(\tq.1^{k}))_{k \ge
    \regi{\iT'}}.$$
  Hence,  by \eqref{eq:33}, we have $\id{\iT'} = \id{\iT(w_r)}$, and so, the equality \eqref{eq:34}
  yields
  \begin{equation}
    \label{eq:57}
    \pi(\iT')  \Req \pi(\iT(w_\ell)) \rho_\h(w_\ell)a \cdot
    \pi(\iT(w_r)).
  \end{equation}

  Now, we need to establish the equality $w_\ell = \pi(\iT(w_\ell))
  \rho_\h(\reg{w_\ell})$. But, using Lemma~\ref{sec:16}, that is
  immediate, since $w_\ell \Req \pi(\iT(w_\ell))
  \reg{\w_\ell}$ modulo $\DRH$ and, by Lemma \ref{l:2},
  $\h$ satisfies $\pi(\iT(w_\ell))\rho_\h(\reg{w_\ell}) =
  \irr{w_\ell}\cdot \id{\iT(w_\ell)} \cdot \reg{w_\ell} = w_\ell$.
  Hence, it follows from \eqref{eq:57} that $w =
  w_\ell \cdot a\cdot w_r \Req \pi(\iT')$, as intended.
  
  The case where $\rho_\DRH(w)$ is regular is handled similarly.
\end{proof}

The \emph{value of a path} $\tq_0 \xrightarrow{\alpha_0} \tq_1
\xrightarrow{\alpha_1} \cdots \xrightarrow{\alpha_n}
\tq_{n+1}$ in a $\DRH$-automaton $\iA$ is given by
$\prod_{i=0}^n \left(\alpha_i, \lambda_{\h, \alpha_i}(\tq_i),
  \lambda(\tq_i)\right) \in \left(\Sigma \times (\pseudo AH)^I \times
  A\right)^+$, where $\lambda_{\h, \alpha_i}(\tq_i)
=\lambda_{\h}(\tq_i)$ if $ \alpha_i = 0$, and  $\lambda_{\h, \alpha_i}(\tq_i)
=I$ otherwise.
Given a state $\tv$ of $\iA$, the
\emph{language associated to $\tv$} is the set $\iL(\tv)$ of all
values of successful paths in
$\iA_{\tv}$. The \emph{language} associated to $\iA$, denoted
$\iL(\iA)$, is the language associated to its root. Finally, the \emph{language
  associated to the pseudoword $w$} is $\iL(w) = \iL(\iT(w))$.

\begin{lem} \label{3.23}
  Let $\iA_1$, $\iA_2$ be $\DRH$-automata. Then, the languages
  $\iL(\iA_1)$ and $\iL(\iA_2)$ coincide if and only if the $\DRH$-trees
  $\vec{\iA}_1$ and $\vec{\iA}_2$ are the same.
\end{lem}
\begin{proof}
  Recall that, by Lemma \ref{3.16}, if $\vec\iA_1 = \vec\iA_2$, then
  $\iA_1$ and $\iA_2$ are equivalent $\DRH$-automata. Hence,
  Definition \ref{d:1} makes clear the reverse implication. Conversely,
  let $\iA_i = \langle V_i, \to_i, \tq_{i,0}, F_i, \lambda_{i,\h},
  \lambda_i \rangle$ ($i = 1, 2$) be two $\DRH$-automata such that
  $\iL(\iA_1)  = \iL(\iA_2)$.
  We first observe that, for $i = 1,2$ and $\alpha \in \Sigma^*$,
  the state $\tq_{i,0}.\alpha$ is defined if and only if there exists an
  element in $\iL(\iA_i)$ of the form $(\alpha, \_, \_)$. Hence,
  the state $\tq_{1,0}.\alpha$ is defined if and only if so is the state
  $\tq_{2,0}.\alpha$.
  Choose $\alpha = \alpha_0\alpha_1\cdots \alpha_n\in \Sigma^*$, with
  each $\alpha_i \in \Sigma$. If
  $\tq_{1,0}.\alpha \in F_1$, then we have a successful path
  $\tq_{1,0} \xrightarrow{\alpha_0} \tq_{1,1}
    \xrightarrow{\alpha_1} \cdots \xrightarrow{\alpha_{n}}
    \tq_{1,n+1}$, so that, the element
  $\prod_{i = 0}^n(\alpha_i,
  \lambda_{1,\h,\alpha_i}(\tq_{1,i}),\lambda_1(\tq_{1,i}))$ belongs to
  $\iL(\iA_1)$ and hence, to $\iL(\iA_2)$. But that implies that, in
  $\iA_2$, there is a successful path  $\tq_{2,0}
    \xrightarrow{\alpha_0} \tq_{2,1} \xrightarrow{\alpha_1} \cdots
    \xrightarrow{\alpha_{n}}
    \tq_{2,n+1},$ which in turn yields that both $\tq_{1,0}.\alpha$
  and $\tq_{2,0}.\alpha$ are terminal states. In particular the equalities in
  \eqref{eq:10} hold. On the other hand, if
  $\tq_{1,0}.\alpha$ is not a terminal state, then condition
  \ref{a3} implies that $\tq_{1,0}.\alpha0$ is defined. Since any
  $\DRH$-automaton is trim, there exists $\beta = \alpha_{n+2} \cdots
  \alpha_{m}  \in \Sigma^*$ such that
  \begin{equation}
    \label{eq:1}
    \tq_{1,0}
      \xrightarrow{\alpha_0} \tq_{1,1} 
      \xrightarrow{\alpha_1} \cdots \xrightarrow{\alpha_{n}}
      \tq_{1,n+1} \xrightarrow{0} \tq_{1,n+2} \xrightarrow{\alpha_{n+2}}
      \cdots \xrightarrow{\alpha_m} \tq_{1,m+1}
  \end{equation}
  is a successful
  path in $\iA_1$. Again, since $\iL(\iA_1) = \iL(\iA_2)$, this
  determines a successful path in $\iA_2$ given~by  $\tq_{2,0}
    \xrightarrow{\alpha_0} \tq_{2,1} 
    \xrightarrow{\alpha_1} \cdots \xrightarrow{\alpha_{n}}
    \tq_{2,n+1} \xrightarrow{0} \tq_{2,n+2} \xrightarrow{\alpha_{n+2}}
    \cdots \xrightarrow{\alpha_m} \tq_{2,m+1},$
  with the same value as the path \eqref{eq:1}. In particular,
  the $(n+2)$-nd letter (of the alphabet $\Sigma \times
  (\pseudo AH)^I \times A$) of that value is
  $$(0, \lambda_{1,\h,0}(\tq_{1,n+1}), \lambda_{1}(\tq_{1,n+1}))=(0,
  \lambda_{2,\h,0}(\tq_{2,n+1}), \lambda_{2}(\tq_{2,n+1})).$$
  But that means
  precisely that the desired equalities in \eqref{eq:10} hold. Therefore,
  $\iA_1$ and $\iA_2$ are equivalent and so, $\vec \iA_1 = \vec\iA_2$.
\end{proof}
\begin{prop} \label{3.24}
  Let $u, v\in \pseudo AS$. Then the equality $\rho_\DRH(u) =
  \rho_\DRH (v)$ holds if and only if $\iL(u) = \iL(v)$ and $\h$
  satisfies $u = v$.
\end{prop}
\begin{proof}
  Let $u$ and $v$ be two equal pseudowords modulo $\DRH$. In particular,
  the $\Req$-classes $[\rho_{\DRH}(u)]_{\Req}$ and
  $[\rho_{\DRH}(v)]_{\Req}$ coincide and so, the $\DRH$-trees $\iT(u)$
  and $\iT(v)$ are the same, by Theorem \ref{3.21}. Therefore, we have
  $\iL(u) = \iL(\iT(u)) = \iL(\iT(v)) = \iL(v)$.
  As $\h$ is a subpseudovariety of $\DRH$, we also have $u =_\h
  v$.
  Conversely, suppose that $\iL(u) = \iL(v)$ and $u =_\h v$.
  By Lemma \ref{3.23}, it follows that $\iT(u) =
  \iT(v)$. Thus, by Theorem~\ref{3.21}, the pseudovariety $\DRH$
  satisfies $u \Req v$. As, in addition, the pseudowords $u$ and $v$
  are equal modulo $\h$, we conclude by Lemma \ref{sec:16} that
  $\DRH$ satisfies $u = v$.
\end{proof}
\section{A canonical form
  for \texorpdfstring{$\kappa$}{$k$}-words over $\DRH$}\label{section4}

Throughout this section, we reserve the letter $\h$ to denote a
pseudovariety of groups 
such that there exists a canonical form for the elements of $\pseudok
AH$.
We denote by $\cf_\h(w)$
the canonical form of $w \in \pseudok
AH$ and set~$\cf_\h(I) = I$.
Our aim is to prove that this assumption on $\h$ is
enough to define a canonical form for the elements of~$\pseudok
A{DRH}$.

Given a finite $\DRH$-automaton $\iA = \langle V, \to \tq, F,
\lambda_\h, \lambda\rangle$ such that $\lambda_\h(V)\subseteq (\pseudok
AH)^I$, let us define the expression
$\picf(\iA)$ inductively on the  number $\card V$ of states as
follows. 
\begin{itemize}
\item If $\card V = 1$, then $\iA = {\bf 1}$ and we take $\picf(\iA) =
  I$.
\item If $\card V > 1$ and $\norm \iA < \infty$, then we put
  \begin{align*}
    \picf(\iA) = \prod_{i = 0}^{\norm{\iA} - 1}
    \picf(\iA_{\tq.1^{i}0}) 
    \cf_\h(\lambda_{\h}(\tq.1^i)) \lambda(\tq.1^i).
  \end{align*}
\item Finally, we suppose that $\card V > 1$ and $\norm \iA = \infty$.
  Since $\iA$ is a finite automaton, we necessarily have a cycle of
  the form $\tq.1^{\ell}\xrightarrow{1} \tq.1^{\ell+1}
  \xrightarrow{1} \cdots \xrightarrow{1} \tq.1^{\ell + n} \xrightarrow{1}
  \tq.1^\ell$, where $\ell$ is a certain integer greater than or equal to
  $\regi \iA$.
  Choose $\ell$ to be the least possible.
  Then, we make $\picf(\iA)$ be given by
  \begin{align*}
    &\prod_{i = 0}^{\regi\iA - 1} \picf(\iA_{\tq.1^{i}0}) 
    \cf_\h(\lambda_{\h}(\tq.1^i)) \lambda(\tq.1^i)
    \\ &  \quad\cdot
    \Bigg( \prod_{i = \regi\iA}^{\ell-1} \picf(\iA_{\tq.1^{i}0})
         \cf_\h(\lambda_{\h}(\tq.1^i)) \lambda(\tq.1^i)
    \\ & \qquad \cdot\Big( \prod_{i = 0}^{n}
    \picf(\iA_{\tq.1^{\ell+i}0}) \cf_\h(\lambda_\h(\tq.1^{\ell + i}))
    \lambda(\tq.1^{\ell +i})\Big)^\omega\Bigg)^\omega.
  \end{align*}
\end{itemize}

We point out that, by definition, the value of the
$\kappa$-word over $\DRH$ naturally induced by $\picf(\iA)$ is precisely
$\pi(\iA)$.
On the other hand, it is easy to check that, for every $w \in \pseudo
A{DRH}$, if $w \Req \pi(\iA)$, then the identity $w = \pi(\iA) \reg w$ holds.
Thus, in view of Theorem \ref{3.21}, we wish to standardize a choice of a finite
$\DRH$-automaton, say $\iA(w)$, equivalent to $\iT(w)$, for each $w
\in \pseudok A{DRH}$.
After that, we may let the canonical form of $w$ be
given by $\picf(\iA(w))\cf_\h(\reg w)$.

Fix a $\DRH$-automaton $\iA = \langle V, \to, \tq, F, \lambda_\h,
\lambda \rangle$. We say that two states $\tv_1, \tv_2 \in V$ are
\emph{equivalent} if $\pi(\iA_{\tv_1})$ and $\pi(\iA_{\tv_2})$ lie in
the same $\Req$-class.
Clearly, this defines an equivalence relation on $V$, say~$\sim$
(it should be clear from the context when we are referring to this
equivalence relation or to the equivalence relation on $\Aa_A$
introduced in Definition \ref{d:1}).
We write $[\tv]$ for the equivalence class of~$\tv \in V$.
\begin{lem}\label{l:1}
  Let $\iA = \langle V, \to, \tq, F, \lambda_\h,\lambda \rangle$ be a
  $\DRH$-automaton and consider the equivalent class on~$V$ defined
  above.
  Then, for every $\tv_1, \tv_2 \in V \setminus F$, we have
  \begin{align*}
    [\tv_1] = [\tv_2] \implies
    \begin{cases}
      [\tv_1.0] = [\tv_2.0] \text{ and } [\tv_1.1] = [\tv_2.1];
      \\ \lambda_\h(\tv_1) = \lambda_\h(\tv_2)
      \text{ and } \lambda(\tv_1) = \lambda(\tv_2).
    \end{cases}
  \end{align*}
\end{lem}
\begin{proof}
  Let $\tv_1, \tv_2 \in V \setminus F$ be non-terminal states.
  By definition, the classes $[\tv_1]$ and~$[\tv_2]$
  coincide if and only if $\pi(\iA_{\tv_1}) \Req \pi(\iA_{\tv_2})$.
  Moreover, by Lemma \ref{l:2}, we have the equality $\lbf{\pi(\iA_{\tv_1})} =
  (\pi(\iA_{\tv_1.0})\lambda_\h(\tv_1), \lambda(\tv_1), w_{1,r})$,
  where $w_{1,r}$ is $\Req$-equivalent to $\pi(\iA_{\tv_1.1})$.
  Similarly, there exists $w_{2,r} \Req \pi(\iA_{\tv_2.1})$ such that
  $\lbf{\pi(\iA_{\tv_2})} =
  (\pi(\iA_{\tv_2.0})\lambda_\h(\tv_2), \lambda(\tv_2), w_{2,r})$.
  In particular, since we are assuming that $\pi(\iA_{\tv_1}) \Req
  \pi(\iA_{\tv_2})$,
  the relations $\pi(\iA_{\tv_1.0}) \Req \pi(\iA_{\tv_2.0})$,
  and $\pi(\iA_{\tv_1.1}) \Req \pi(\iA_{\tv_2.1})$ hold.
  But, that means that~$[\tv_1.0] = [\tv_2.0]$ and~$[\tv_1.1] =
  [\tv_2.1]$.
  Also, the mid components of $\lbf{\pi(\iA_{\tv_1})}$ and
  $\lbf{\pi(\iA_{\tv_2})}$ should coincide, that is, $ \lambda(\tv_1)
  = \lambda(\tv_2)$.
  Finally, we may derive the equality $\lambda_\h(\tv_1) =
  \lambda_\h(\tv_2)$ as follows:
  \begin{align*}
    & \pi(\iA_{\tv_1.0})\lambda_\h(\tv_1) =
    \pi(\iA_{\tv_2.0})\lambda_\h(\tv_2) \quad \text{ because
    $\pi(\iA_{\tv_1}) \Req \pi(\iA_{\tv_2})$}
    \\ \iff & \piir{\iA_{\tv_1.0}}\id{\iA_{\tv_1.0}}\lambda_\h(\tv_1) =
    \piir{\iA_{\tv_2.0}}\id{\iA_{\tv_2.0}}\lambda_\h(\tv_2)
              \quad\text{ by \eqref{eq:32}}
    \\ \implies &\id{\iA_{\tv_1.0}}\lambda_\h(\tv_1) =
    \id{\iA_{\tv_2.0}}\lambda_\h(\tv_2)\quad \text{ by Lemma \ref{l:2}
    and Proposition \ref{p:1}}
    \\ \implies &\lambda_\h(\tv_1) = \lambda_\h(\tv_2).\qed \popQED
  \end{align*}
\end{proof}
We define the \emph{wrapping of a $\DRH$-automaton}
$\iA = \langle V,
\to, \tq, F, \lambda_\h,\lambda \rangle$ to be the $\DRH$-automaton
$[\iA] = \langle V/{\sim}, \to, [\tq], F/{\sim},
\overline\lambda_\h,\overline\lambda \rangle$, where
\begin{itemize}
\item $[\tv].0 = [\tv.0]$ and $[\tv].1 = [\tv.1]$, for $\tv \in V
  \setminus F$;
\item $\overline\lambda_\h([\tv]) = \lambda_\h(\tv)$ and
  $\overline\lambda([\tv]) = \lambda(\tv)$, for $\tv \in V$.
\end{itemize}
By Lemma \ref{l:1}, this automaton is well defined.
Furthermore, its definition ensures that~$\iA \sim [\iA]$.
The
\emph{wrapped $\DRH$-automaton of $w \in \pseudo A{DRH}$} is $\iA(w) =
[\iT(w)]$.
Observe that, by Lemmas \ref{l:6} and \ref{arv}, the label
$\lambda_\h$ of $\iT(w)$ takes values in $\pseudok AH$ when $w$ is a
$\kappa$-word.
Our next goal is to prove that $\iA(w)$ is finite, provided $w$ is a
$\kappa$-word.

Let us
associate to a pseudoword $w \in (\pseudo A{DRH})^I$ a certain set of
its factors.
For $\alpha \in \Sigma^*$, we define $f_\alpha(w)$ inductively on
$\card \alpha$:
\begin{align*}
  f_\varepsilon(w) &= w;
  \\ (f_{\alpha0}(w), a, f_{\alpha1}(w)) &= \lbf {f_\alpha(w)}, \text{
  for a certain $a \in A$, whenever $f_\alpha(w) \neq I$}.
\end{align*}
Then, the set of \emph{$\DRH$-factors}
of $w$ is given by
$$\iF(w) = \{f_\alpha(w)\colon \alpha \in \Sigma^* \text{ and }
f_\alpha(w) \text{ is defined}\}.$$
The relevance of the definition of the set $\iF(w)$ is explained
by the following result.
\begin{lem}\label{l:3}
  Let $w \in \pseudo A{DRH}$ and $\iT(w) = \langle V, \to, \tq, F,
  \lambda_\h, \lambda\rangle$. Then, for every $\alpha \in \Sigma^*$
  such that $f_\alpha(w)$ is defined,
  the relation~$f_\alpha(w) \Req \pi(\iT(w)_{\tq.\alpha})$ holds.
\end{lem}
\begin{proof}
  We prove the statement by induction on $\card \alpha$.
  When $\alpha = \varepsilon$, the result follows from Theorem~\ref{3.21}.
  Let $\alpha \in \Sigma^*$ and invoke the induction hypothesis to assume
  that
  $f_\alpha(w)$ and $\pi(\iT(w)_{\tq.\alpha})$ are $\Req$-equivalents.
  Writing $\lbf{\pi(\iT(w)_{\tq.\alpha})} =
  (w_\ell,a,w_r)$,
  Lemma~\ref{l:2} yields the relations
  $w_\ell \Req \pi(\iT(w)_{\tq.\alpha0})$ and $w_r \Req
  \pi(\iT(w)_{\tq.\alpha1})$.
  On the other hand, since $\lbf{f_\alpha(w)} = (f_{\alpha0}(w), b,
  f_{\alpha1}(w))$,
  using Lemma \ref{sec:13} we deduce that $f_{\alpha0} = w_\ell$, $a =
  b$, and $f_{\alpha1} \Req w_r$, leading to the desired conclusion.
\end{proof}

Hence, in order to prove that $\iA(w)$ is finite for every $\kappa$-word
$w$, it suffices to prove that so is~$\iF(w)/{\Req}$.
The next two lemmas are useful to achieve that target.

\begin{lem}\label{l:13}
  Let $w$ be a regular $\kappa$-word over $\DRH$.
  Then, there exist $\kappa$-words $x$, $y$ and $z$ over $\DRH$ such
  that $w = xy^{\omega-1} z$, $c(y) = c(w)$, $\cum x \subsetneqq c(
  w)$, and $y$ is not regular.
\end{lem}
\begin{proof}
  By definition of $\kappa$-word, we may write $w = w_1\cdots w_n$,
  where each $w_i$ is either a letter in $A$ or an $(\omega-1)$-power
  of another $\kappa$-word.
  Since any letter of the cumulative content of $w$
  occurs in~$\lbf[\infty] w$
  infinitely many times,
  there must be an $(\omega-1)$-power under which they all appear.
  Hence, since $w$ is regular (and so, $c(w) = \cum w$),
  there exists an index $i \in \{1,\ldots, n\}$ such that $w_i
  = v^{\omega-1}$ and $c(v) = c(w)$.
  Let $j$ be the minimum such~$i$.
  We have $w = u_0 v_0^{\omega-1}z_0$, where $u_0 = w_1\cdots w_{j-1}$,
  $v_0^{\omega-1} = w_j$, and~$z_0 = w_{j+1}\cdots w_n$.
  Also, minimality of $j$ yields that $\cum{u_0}\subsetneqq \cum w = c(w)$.
  So, if $v_0$ is not regular, then we just take $x = u_0$, $y
  = v_0$, and $z = z_0$.
  Suppose that $v_0$ is regular.
  Using the same reasoning, we may write $v_0 = u_1
  v_1^{\omega-1}z_1$, with $\cum{u_1} \subsetneqq c(w)$ and $c(v_1)
  = c(v_0) = c(w)$.
  Again, if $v_1$ is not regular, then we may choose $x =
  u_0u_1$, $y = v_1$ and $z = z_1 v_0^{\omega-2} z_0$.
  Otherwise, we repeat the process with $v_1$.
  Since $w$ is a $\kappa$-word, there is only a finite number of
  occurrences of $(\omega-1)$-powers, so that, this iteration cannot
  run forever.
  Therefore, we eventually get $\kappa$-words $x$, $y$ and $z$
  satisfying the desired properties.
\end{proof}
\begin{lem}\label{l:14}
  Let $w \in \pseudok A{DRH}$ be regular.
  For each $m \ge 1$, let $w_m'$ be the unique $\kappa$-word over $\DRH$
  satisfying the equality $w = \lbf[1] w \cdots \lbf[m] w w_m'$.
  Then, both sets $\{\lbf[m] w \colon m \ge 1\}$ and
  $\{[w_m']_{\Req}\colon m \ge 1\}$ are finite.
\end{lem}
\begin{proof}
  Write $\lbf[m] w = w_ma_m$, for every $m \ge 1$, and $w =
  xy^{\omega-1}z$, with $x$, $y$ and $z$ satisfying
  the properties stated in Lemma~\ref{l:13}.
  We define a sequence of pairs of possibly empty $\kappa$-words
  $\{(u_i,v_i)\}_{i \ge 0}$ and a strictly increasing sequence of
  non-negative integers $\{k_i\}_{i \ge 0}$ inductively as follows.
  We start with $(u_0, v_0) = (I, x)$ and we let $k_0$ be the maximum
  index such that $\lbf[1] w\cdots \lbf[{k_0}] w$ is a prefix of $x$.
  If $x$ has no prefix of this form, then we set $k_0 = 0$.
  We also write $v_0 = v_0'v_0''$, with
  $v_0 ' = \lbf[1] w\cdots \lbf[k_0]w$ (by Proposition \ref{p:1}, given
  $v_0'$ there is only one possible value for $v_0''$).
  For each $i \ge 0$, we let $u_{i+1}$ be such that $w_{k_i+1} = v_i''
  u_{i+1}$ and $v_{i+1}$ is such that $y = u_{i+1}a_{k_i+1}v_{i+1}$.
  Observe that, by uniqueness of first-occurrences
  factorizations, there is
  only one pair $(u_{i+1}, v_{i+1})$ satisfying these conditions.
  The integer $k_{i+1}$ is the maximum such that $\lbf[{k_i+2}] w\cdots
  \lbf[k_i+1] w$ is a prefix of $v_{i+1}$ (or $k_{i+1} = k_i+1$ if
  there is no such prefix) and we factorize
  $v_{i+1} = v_{i+1}'v_{i+1}''$, with $v_{i+1}' = \lbf[{k_i+2}] w\cdots
  \lbf[{k_{i+1}}] w$.
  By construction, for all $i \ge 0$, the
  pseudoidentity 
  $w_{k_i+1}' = v_{i+1} y^{\omega-(i+2)}z$
  holds.
  In particular, for every $m \ge 1$, there exist $i \ge 0$ and $\ell
  \in \{2, \ldots, k_{i+1} - k_i\}$ such that
  \begin{equation}\label{eq:62}
    w_m' = \lbf[k_i+\ell]w  \lbf[{k_i+\ell+1}]w \cdots  \lbf[{k_{i+1}}]
    w v_{i+1}'' y^{\omega-(i+2)}z.
  \end{equation}
  
  On the other hand, for all $i \ge 0$, the factorization
  $y = u_{i+1}a_{k_i+1}v_{i+1}$ is such that $a_{k_i+1} \notin
  c(u_{i+1})$ (recall that $a_{k_i+1} \notin c(w_{k_i+1})$ and
  $u_{i+1}$ is a factor of $w_{k_i+1}$).
  By uniqueness of first-occurrences factorization over $\DRH$,
  it follows that the set $\{(u_{i},v_{i})\}_{i \ge 0}$ is finite.
  Consequently, the set
  $\{\lbf[{k_i+\ell}]w  \lbf [{k_i+\ell+1}]w \cdots  \lbf[{k_{i+1}}] w
  v_{i+1}'' \colon i \ge 0, \: \ell \in \{2, \ldots, k_{i+1}-k_i\}\}$
  is also finite.
  In particular, there is only a finite number of $\kappa$-words $\lbf[m] w$.
  Finally, taking into account that $c(z) \subseteq c(y)$ and
  \eqref{eq:62} we may conclude that there are only finitely many
  $\Req$-classes of the form~$[w_m']_{\Req}$~($m \ge 1$).
\end{proof}

Now, we are able to prove that $\iF(w)/{\Req}$ is finite for every
$\kappa$-word $w$ over $\DRH$.

\begin{prop}\label{p:10}
  Let $w$ be a possibly empty $\kappa$-word over $\DRH$. Then, the quotient
  $\iF(w)/{\Req}$ is finite.
\end{prop}
\begin{proof}
  We prove the result by induction on $\card{c(w)}$.
  If $\card{c(w)} = 0$, then it is trivial.
  Suppose that $\card{c(w)} \ge 1$.
  We distinguish two possible scenarios.
  \begin{description}
  \item[Case 1.] The $\kappa$-word $w$ is not regular, that is, $\cum
    w \subsetneqq c(w)$.

    Then, there exists $k \ge 1$ such that $w = w_1a_1 \cdots w_ma_m
    w_m'$, with $\lbf[k] w = w_ka_k$, for $k = 1, \ldots, m$ and
    $c(w_m') \subsetneqq c(w)$.
    By definition of $f_\alpha(w)$, we have the identities
    $f_{1^{k-1}0} = w_k$ (for $k = 1, \ldots, m$) and $f_{1^m} = w_m'$.
    Hence, we may deduce that $\iF(w)$ is the union of the sets
    $\iF(w_k)$ (for $k = 1, \ldots, m$) together with $\iF(w_m')$.
    Using the induction hypothesis on each one of the intervening sets, we
    conclude that $\iF(w)$ is finite.
  \item[Case 2.] The $\kappa$-word $w$ is regular.

    Again, write $\lbf[k] w = w_ka_k$ and $w = \lbf[1]
    w\cdots \lbf[k] w w_k'$, for $k \ge 1$.
    Since $f_{1^{k-1}0}(w) = w_k$ and $f_{1^{k}}(w) = w_{k}'$, for every
    $k \ge 1$, by Lemma \ref{l:14},
    we know that the sets $\{f_{1^{k-1}0}(w)\}_{k \ge 1}$
    and~$\{[f_{1^{k}}(w)]_{\Req}\}_{k \ge 1}$ are both finite. 
    Applying the induction hypothesis to each factor $w_k$, we derive
    that~$\{[f_{1^{k-1}0\alpha}(w)]_{\Req}\colon \alpha \in
    \Sigma^*\}_{k \ge 1}$ is also a finite set.
    Therefore, since any element of $\iF(w)/{\Req}$ is of one of the
    forms ${[}f_{1^{k-1}0\alpha}(w)]_{\Req}$ and $[f_{1^{k}}(w)]_{\Req}$,
    we conclude that $\iF(w)/{\Req}$ is finite as well.\popQED
  \end{description}
\end{proof}

As an immediate consequence (recall Lemma \ref{l:3}), we obtain:
\begin{cor}\label{c:14}
  Let $w$ be a possibly empty $\kappa$-word. Then, the wrapped
  $\DRH$-automaton $\iA(w)$ is finite. \qed
\end{cor}

Unlike the aperiodic case $\R$, the converse of Corollary
\ref{c:14} does not hold in general.
For instance, taking $\h = \G$,
it is not hard to see that $\iA(a^{p^\omega}b)$
(with $p$ a prime number) is finite, although $a^{p^\omega}b$ is not a
$\kappa$-word over ${\sf DRG}$.
A converse is achieved when we further require that the labels
$\lambda_\h$ are valued by $\kappa$-words over $\h$ and that
$\rho_\h(\reg w)$ is
itself a $\kappa$-word.

For a given $w \in (\pseudok A{DRH})^I$, the
expression $$\cf (w) =
\picf(\iA(w)) \cf_\h(\rho_\h(\reg w))$$ is said the
\emph{canonical form of
  $w$}.
We write $\cf(u) \equiv \cf(v)$  (with $u,v \in (\pseudok A{DRH})^I$)
when both sides coincide.
We have just proved the claimed existence of a canonical form for
elements of~$\pseudok A{DRH}$.

\begin{thm}
  Let $\h$ be a pseudovariety of groups such that there exists a
  canonical form for the elements of $\pseudok AH$, say $\cf_\h(\_)$. Then, for all
  $\kappa$-words $u$ and $v$ over $\DRH$, the equality $u = v$ holds if
  and only if $\cf (u) \equiv \cf(v)$. \qed
\end{thm}
\section{\texorpdfstring{$\overline\kappa$}{\textbackslash overline k}-terms seen as
  well-parenthesized words}\label{section5}

In Section \ref{section3}, we characterized $\Req$-classes over $\DRH$ by
means of certain equivalence classes of automata.
In order to solve the $\kappa$-word problem
over $\DRH$, the next goal is to find an algorithm to construct such
automata. This section serves the purpose of preparing that
construction.

\subsection{General definitions}

Let $B$ be a possibly infinite alphabet and
consider the associated alphabet $B_{[\:]} = B \uplus \{[^q,]^q \colon
q \in \zz\}$. We say that a word in $B_{[\:]}^*$ is
\emph{well-parenthesized over $B$} if it does not contain $[^q\:]^q$
as a
factor and if it can be reduced to the empty word $\varepsilon$ by
applying the rewriting rules $[^q \:]^q \to \varepsilon$ and
$a \to \varepsilon$, for $q \in \zz$ and $a \in B$.
We denote the set of all
well-parenthesized words over $B$ by $\dyck B$. The \emph{content} of
a
well-parenthesized word~$x$ is the set of letters in $B$ that occur in
$x$ and it is denoted $c(x)$.

To each $\overline\kappa$-term we
may associate a well-parenthesized word over $A$ inductively as
follows:
\begin{align*}
  \word a &= a, \quad\text{ if } a \in A;
  \\  \word {u \cdot v} &= \word u  \word v, \quad\text{ if $u$ and $v$ are $\overline\kappa$-terms};
  \\  \word {u^{\omega + q}}& = [^q \word u]^q, \quad \text{ if $u$ is
                              a $\overline\kappa$-term}.
\end{align*}
Conversely, we associate a $\kappa$-word
to each well-parenthesized word over $A$ as follows:
\begin{align*}
 \om a  & = a, \quad\text{ if } a \in A;
  \\  \om {xy}  & = \om x \cdot \om y, \quad\text{ if } x,y \in \dyck A;
  \\  \om {[^q x]^q}  & = \om x^{\omega + q}, \quad \text{ if } x \in \dyck A.
\end{align*}
Note that, due to the associative property in both $\dyck A$ and
$\pseudok AS$, $\om\_$ is well-defined.
With the aim of distinguishing the occurrences of each letter in $A$ in a
well-parenthesized word $x$ over $ A$, we assign to each $x \in \dyck
A$ a well-parenthesized word $x_\nn$ over $A \times \nn$ 
containing all
the information about the position of the letters. With that in mind
we define recursively the following family of functions $\{p_k:\dyck A \to \dyck {A \times
  \nn}\}_{k \ge 0}$:
\begin{align*}
  p_k(a) &= (a, k+1), \quad\text{ if } a \in A;
  \\  p_k([^q) &= {[}^q \text{ and } p_k(]^q) = {]}^q,\quad \text{ if } q \in \zz;
  \\ p_k(ay) &= p_k(a)p_{k+1}(y), \quad\text{ if } a \in A_{[\:]} \text{ and } y
  \in A_{[\:]}^*.
\end{align*}
We define $x_\nn = p_0(x)$. For instance, if $x = a[^qb[^rc{]}^r{]}^q$,
then $x_\nn$ is the word $(a,1) [^q(b,2)[^r(c,3)]^r]^q$. It is often
convenient to denote the pair $(a,i)$ by $a_i$.
Let $x \in \dyck{A \times \nn}$. Then, we may associate to $x$ two
well-parenthesized words $\pi_A(x)$ and $\pi_\nn(x)$ corresponding to
the projection of $x$ onto $A_{[\:]}^*$ and onto $\nn_{[\:]}^*$,
respectively. We denote $c_A(x) = c(\pi_A(x))$ and $c_\nn(x) = c(\pi_\nn(x))$.
Given a $\overline\kappa$-term $w$, we denote by $\w$ the well-parenthesized
word $0_0\word {w\#}_\nn$ over the
alphabet $(A \uplus \{0, \#\})\times \nn$.
The map $\eta:\dyck{A\times \nn} \to \pseudok AS$ assigns to each
well-parenthesized word $x \in \dyck{A \times \nn}$ the
$\kappa$-word $\eta(x) = \om{\pi_A(x)}$.

Let $x$ be a well-parenthesized word over $A \times \nn$. We define its
\emph{tail} $\Tt_i(x)$ from position $i \in \nn$ inductively as
follows
\begin{align*}
  \Tt_i(\varepsilon) &= \varepsilon; \\
  \Tt_i(yz) &= \Tt_i(z), \quad\text{ if } y, z \in \dyck{A\times \nn}
    \text{ and } i \notin c_{\nn}(y); \\
 \Tt_i(a_iy) &= y,  \quad\text{ if } y \in \dyck{A \times \nn}; \\
  \Tt_i([^q y ]^q z) &= \Tt_i(y) [^{q-1} y]^{q-1} z, \quad \text{ if
    } y, z \in \dyck{A \times \nn} \text{ and } i \in c_{\nn}(y).
\end{align*}
The \emph{prefix} of $x \in \dyck{A \times \nn}$ until $a \in A$ is
defined by
\begin{align*}
   \tp_a(\varepsilon) &= \varepsilon;
  \\  \tp_a(y z) &= y \tp_a(z),  \quad\text{ if } y, z \in \dyck{A \times
    \nn} \text{ and } a \notin c_A(y);
  \\  \tp_a(a_iy) &= \varepsilon,  \quad\text{ if } y \in \dyck{A \times
    \nn};
  \\ \tp_a([^q y ]^q z) &= \tp_a(y),  \quad\text{ if } y, z \in \dyck{A
    \times \nn} \text{ and } a \in c_A(y).
\end{align*}
The \emph{factor} of a well-parenthesized word $x \in \dyck{A\times \nn}$ \emph{from
$i\in \nn$ until $a \in A$} is given~by
$$x(i,a) = \tp_a(\Tt_i(x)).$$
If instead, we are given a $\overline\kappa$-term $w$, then we write $w(i,a)$
to mean the $\kappa$-word $\eta(\w(i,a))$.
If $a$ is a letter occurring in $\pi_A(x)$, for a well-parenthesized word $x$ over $A
\times \nn$, then it is possible to write $x = y  a_i z$
with $y$ and $z$ possibly empty not necessarily well-parenthesized
words over $A \times \nn$ such that $a \notin c_A(y)$. In this case we say that $a_i$ is a
\emph{marker of $x$}. If $a_i$ is the last first occurrence of a
letter, that is, if the inclusion $c_A(z) \subseteq c_A(y a_i)$
holds, then we say that $a_i$ is the \emph{principal marker of $x$}.

\subsection{Properties of tails and prefixes of well-parenthesized words}

The next results state some properties concerning tails and prefixes of
well-parenthesized words.
Some of the proofs are omitted since they are rather technical and
entirely similar to the proofs of the analogous results in \cite{palavra}.
When that is the case, we refer the reader to the corresponding
result.
\begin{lem}[cf.\ {\cite[Lemma 5.3]{palavra}}]\label{5.3}
  Let $x \in \dyck{A \times \nn}$ and let $a,b \in A$.
  If $b \in c_A(\tp_a(x))$, then $\tp_b(\tp_a(x)) = \tp_b(x)$.
\end{lem}
\begin{lem}[cf.\ {\cite[Lemma 5.4]{palavra}}]\label{5.4}
  Let $x \in \dyck{A \times \nn}$ be such that $a$ belongs to $c_A(x)$. If~$k
  \in c_{\nn}(\tp_a(x))$, then $a \in c_A(\Tt_k(x))$. 
\end{lem}
\begin{lem}[cf.\ {\cite[Lemma 5.5]{palavra}}]\label{5.5}
  Let $x \in \dyck{A \times \nn}$ and let $k \in
  c_{\nn}(\tp_a(x))$. Then, we have $\Tt_k(\tp_a(x)) =
  \tp_a(\Tt_k(x))$. 
\end{lem}
\begin{lem}\label{5.7comp}
  Let $\vx = (x_j)_{j\ge 0}$ and $\vy = (y_j)_{j\ge 0}$ be two
  sequences of possibly empty well-parenthesized words over $A \times
  \nn$ such that $x_0y_0 \neq \varepsilon$, and for every $i,j \ge 0$,
  the index $i$ occurs in
  $\pi_\nn(x_0y_0x_1y_1\cdots x_jy_j)$ at most once.
  Let $\vq = (q_j)_{j \ge 0}$ be a sequence of integers. For
  each $n \ge 0$, we define the well-parenthesized words
  $\mu_n(\vx, \vy, \vq)$ and $\xi_n(\vx, \vy, \vq)$ as follows:
  \begin{align*}
     \mu_0(\vx,\vy, \vq) &= x_0y_0 \\
     \mu_{n+1}(\vx, \vy, \vq) &= x_{n+1} [^{q_n} \mu_n(\vx, \vy, \vq)]^{q_n} y_{n+1},
      \text{ if } n \ge 0 \\
     \xi_n(\vx, \vy, \vq) &= [^{q_n-1} \mu_n(\vx, \vy, \vq)]^{q_n-1} y_{n+1},
      \text{ if } n \ge 0.
  \end{align*}
  Let $i$ be a natural number and suppose that $i \in c_{\nn}(x_\ell
  y_\ell)$ for a certain $\ell \ge 0$. Then, for every $n \ge \ell$, the
  following equality holds:
  \begin{align}
    \label{eq:5}
    \Tt_i(\mu_n(\vx,\vy, \vq)) = \Tt_i (\mu_\ell(\vx, \vy, \vq)) \cdot
    \xi_\ell(\vx, \vy, \vq)  \cdot \xi_{\ell+1}(\vx, \vy, \vq) \cdots
    \xi_{n-1}(\vx, \vy, \vq).
  \end{align}
\end{lem}
\begin{proof}
  We argue by induction on $n$. If $n = \ell$, then the result holds
  clearly, since the factor $\xi_\ell(\vx, \vy, \vq) \cdot \xi_{\ell+1}(\vx, \vy, \vq) \cdots
    \xi_{n-1}(\vx, \vy, \vq)$ vanishes in \eqref{eq:5}.
  Suppose that $n > \ell$ and that the result holds for any smaller
  $n$. We may compute
  \begin{align*}
    \Tt_i(\mu_n(\vx, \vy, \vq)) &= \Tt_i(x_n[^{q_{n-1}} \mu_{n-1}(\vx,
                                  \vy, \vq)]^{q_{n-1}} y_n)
    \\ & = \Tt_i(\mu_{n-1}(\vx,\vy, \vq)) \cdot [^{q_{n-1}-1}
         \mu_{n-1}(\vx, \vy, \vq)]^{q_{n-1}-1} y_n
    \\ & \hspace{3.5cm} \text{ since } i \notin c_{\nn}(x_n)\text{ and } i \in c_{\nn}(\mu_{n-1}(\vx,\vy,\vq))
    \\ & = \Tt_i(\mu_{n-1}(\vx,\vy, \vq)) \cdot \xi_{n-1}(\vx,\vy,\vq)
    \\ & =  \Tt_i (\mu_\ell(\vx, \vy, \vq)) \hspace{3cm}\text{ by induction
         hypothesis}
    \\ & \quad \cdot \xi_\ell(\vx, \vy, \vq) \cdots \xi_{n-2}(\vx,
         \vy, \vq) \cdot \xi_{n-1}(\vx,\vy,\vq)
  \end{align*}
  obtaining the desired equality \eqref{eq:5}.
\end{proof}

By successively applying Lemma~\ref{5.7comp}, we obtain the following:

\begin{cor}
  \label{c:1}
  Using the same notation and assuming the same hypothesis
  as in the previous lemma, suppose that $k \in c_\nn(y_0)$. Then,
  \begin{enumerate}
  \item\label{item:1} if $i \in c_\nn(x_\ell)$ for a certain $\ell \ge 0$, then the
    equality
    $$\Tt_k(\Tt_i(\mu_n(\vx, \vy, \vq))) = \Tt_k(y_0)\cdot \xi_0(\vx,
    \vy, \vq)\cdot \xi_1(\vx, \vy, \vq)\cdots \xi_{n-2}(\vx, \vy,
    \vq)\cdot \xi_{n-1}(\vx, \vy, \vq)$$
    holds for every $n \ge \ell$;
  \item\label{item:2} if $i \in c_\nn(y_\ell)$ for a certain $\ell \ge 1$, then the
    equality
    \begin{align*}
    \Tt_k(\Tt_i(\mu_n(\vx, \vy, \vq)))
    & = \Tt_k(y_0)\cdot \xi_0(\vx, \vy, \vq)\cdot \xi_1(\vx, \vy,
      \vq)\cdots \xi_{\ell-1}(\vx, \vy, \vq)
      \\ & \quad \cdot [^{q_\ell-2}\mu_\ell(\vx, \vy,
  \vq)]^{q_\ell-2}y_{\ell+1}\cdot \xi_{\ell+1}(\vx, \vy, \vq)\cdots
  \xi_{n-1}(\vx, \vy, \vq)
  \end{align*}
    holds for every $n \ge \ell$.\qed
  \end{enumerate}
\end{cor}
The reader may wish to compare the next result with \cite[Lemma 5.8]{palavra}.
\begin{lem}\label{5.8}
  Let $w$ be a $\overline\kappa$-term, $i \ge 0$, and $a \in c(w)$. Assume
  that $b_k$ is the principal marker of $\overline{w}(i,a)$. Then, the
  following properties hold:
  \begin{enumerate}
  \item $\tp_b(\overline{w}(i,a)) = \overline{w}(i,b)$;
  \item\label{wia} $\DRH$ satisfies $\eta(\Tt_k(\overline{w}(i,a))) \Req w(k,a)$.
  \end{enumerate}
  Moreover, if the projection of $w(i,a)$ onto $\pseudo A{DRH}$ is not
  regular, then the relation in \ref{wia} becomes an equality in
  $\pseudo AS$.
\end{lem}
\begin{proof}
  By definition, we have $\overline{w}(i,a) =
  \tp_a(\Tt_i(\overline{w}))$. Since $b 
  \in c_A(\overline{w}(i,a))$, it follows from Lemma \ref{5.3} that
  $\tp_b(\overline{w}(i,a)) = \tp_b(\tp_a(\Tt_i(\overline{w}))) =
  \tp_b(\Tt_i(\overline{w})) = \overline{w}(i,b)$.

  Let us prove the second assertion.
  By definition of $\overline{w}$, we
  know that $b_k$ appears exactly once in $\overline{w}$ and the same
  happens with the index $i$. Let $\w = x \cdot b_k \cdot y$. We
  distinguish the cases where $x$ and $y$ are both possibly empty
  well-parenthesized words and where neither of $x$ nor $y$ is a
  well-parenthesized word.
  In the first case, since $b_k \in c(\overline{w}(i,a)) \subseteq
  c(\Tt_i(\overline{w}))$, the index $i$ must belong to
  $c_{\nn}(x)$. So, we get
  $\Tt_k(\overline{w}(i,a)) = \Tt_k(\tp_a(\Tt_i(\overline{w}))) =
  \Tt_k(\tp_a(\Tt_i(x)b_ky))$. Should $a$ occur in $\Tt_i(x)b_k$,
  then $b_k$ would not appear in $\overline{w}(i,a)$. So, it follows
  that
  \begin{align}
    \label{eq:4}
    \Tt_k(\tp_a(\Tt_i(x)b_ky)) = \Tt_k(\Tt_i(x)b_k\tp_a(y)) =
    \tp_a(y).
  \end{align}
  On the other hand, we have the equalities
  $\overline{w}(k,a) = \tp_a(\Tt_k(\overline{w})) = \tp_a(y) \just
  = {\eqref{eq:4}} \Tt_k(\w(i,a))$,
  and so the desired relation \ref{wia} follows.
  
  Now, we suppose that
  \begin{align*}
    x & = x_n [^{q_{n-1}} x_{n-1} \cdots
        [^{q_1}x_1 [^{q_0} x_0,
    \\ b_k y & = y_0]^{q_0}y_{1}]^{q_{1}}\cdots
               y_{n-1}]^{q_{n-1}} y_n,
  \end{align*}
  where all the $x_j$'s and $y_j$'s are
  possibly empty
  well-parenthesized words, for $j = 0, \ldots, n$.
  We note that, since $k \in c_{\nn}(\w(i,a)) =
  c_{\nn}(\tp_a(\Tt_i(\w)))$, Lemma \ref{5.5} yields the equalities
  \begin{equation}
    \Tt_k(\w(i,a)) = \Tt_k(\tp_a(\Tt_i(\w))) =
    \tp_a(\Tt_k(\Tt_i(\w))).\label{eq:11}
  \end{equation}
  With that in mind, we start by computing the elements $\Tt_k(\w)$ and
  $\Tt_k(\Tt_i(\w))$.
  Let
  \begin{align*}
    \vx &= (x_0, x_1, \ldots, x_n, \varepsilon, \varepsilon, \ldots);
    \\  \vy &= (y_0, y_1, \ldots, y_n, \varepsilon, \varepsilon, \ldots );
    \\  \vq &= (q_0, q_1, \ldots, q_{n-1}, 0, 0, \ldots)
  \end{align*}
  and let $\ell \in \{0, 1, \ldots, n\}$ be such that $i \in
  c_{\nn}(x_\ell y_\ell)$.
  Noticing that $\w = \mu_n(\vx,\vy, \vq)$, $k$ belongs to $ c_{\nn}(y_0)$, and
  using Lemma \ref{5.7comp}  we obtain
  \begin{align}
    \Tt_k(\w) &= \Tt_k (\mu_0(\vx, \vy, \vq)) \cdot
                \xi_0(\vx, \vy, \vq)  \cdot \xi_{1}(\vx, \vy, \vq) \cdots
                \xi_{n-1}(\vx, \vy, \vq)\nonumber
    \\ & = \Tt_k(y_0) \cdot \xi_0(\vx, \vy, \vq)  \cdot \xi_{1}(\vx,
         \vy, \vq) \cdots \xi_{n-1}(\vx, \vy, \vq)\label{eq:45}
  \end{align}
  Now, we have two possible situations.
  \begin{enumerate}[label = (\roman*)]
  \item \label{xl} $i \in c_{\nn}(x_\ell)$, for a certain $\ell
    \in \{0, \ldots, n\}$;
  \item \label{yl} $i \in c_{\nn}(y_\ell)$, for a certain $\ell
    \in \{n, \ldots, 1\}$.
  \end{enumerate}
  If we are in Case \ref{xl}, then we may use Corollary \ref{c:1}\ref{item:1} and
  get
  \begin{align*}
    \Tt_k(\Tt_i(\w)) = \Tt_k(y_0) \cdot \xi_0(\vx, \vy, \vq)  \cdot
    \xi_{1}(\vx, \vy, \vq) \cdots  \xi_{n-2}(\vx, \vy, \vq)
    \cdot \xi_{n-1}(\vx, \vy, \vq).
  \end{align*}
  Hence, we have an equality between
  $\Tt_k(\w(i,a)) =\tp_a(\Tt_k(\Tt_i(\w)))$ and $\w(k,a) =
  \tp_a(\Tt_k(\w))$, thereby proving \ref{wia}.
  
  On the other hand, when the situation occurring is \ref{yl}, 
  Corollary \ref{c:1}\ref{item:2} yields
  \begin{align*}
    \Tt_k(\Tt_i(\w))
    & = \Tt_k(y_0) \cdot \xi_0(\vx, \vy, \vq)  \cdot
      \xi_{1}(\vx, \vy, \vq) \cdots  \xi_{\ell-1}(\vx, \vy, \vq)
    \\ & \quad \cdot [^{q_{\ell}-2} \mu_{\ell}(\vx, \vy,
         \vq)]^{q_{\ell}-2}  y_{\ell+1}  \cdot
         \xi_{\ell+1}(\vx, \vy, \vq) \cdots  \xi_{n-1}(\vx, \vy, \vq).
  \end{align*}
  If the first occurrence of $a$ in $\Tt_k(\Tt_i(\w))$ is in
  $$\Tt_k(y_0) \cdot \xi_0(\vx, \vy, \vq)  \cdot \xi_{1}(\vx, \vy,
  \vq) \cdots  \xi_{\ell-1}(\vx, \vy, \vq)$$ or in $\mu_{\ell}(\vx, \vy,
  \vq)$, then the first occurrence of $a$ in $\Tt_k(\w)$ is also in
  one of these factors and we easily conclude that
  \begin{align*}
    \tp_a(\Tt_k(\Tt_i(\w)))
    & = \tp_a(\Tt_k(y_0) \cdot \xi_0(\vx, \vy,
      \vq)  \cdot \xi_{1}(\vx, \vy, \vq) \cdots  \xi_{\ell-1}(\vx, \vy,
      \vq)\cdot \mu_{\ell}(\vx, \vy, \vq))
    \\ & = \tp_a(\Tt_k(\w)),
  \end{align*}
  thereby proving again an equality in \ref{wia}.
  
  Otherwise, the first occurrence of $a$ in $\Tt_k(\Tt_i(\w))$ is in
  $$y_{\ell+1}  \cdot \xi_{\ell+1}(\vx, \vy, \vq) \cdots  \xi_{n-1}(\vx,
  \vy, \vq).$$
  Analyzing the equality \eqref{eq:45}, we deduce that $a$
  occurs for the first time in $\Tt_k(\w)$ also in the factor
  $y_{\ell+1}  \cdot \xi_{\ell+1}(\vx, \vy, \vq) \cdots  \xi_{n-1}(\vx,
  \vy, \vq)$.
  Then, we may compute
  \begin{align}
    \tp_a(\Tt_k(\Tt_i(\w)))  &= \Tt_k(y_0) \cdot \xi_0(\vx, \vy,
      \vq)  \cdot \xi_{1}(\vx, \vy, \vq) \cdots  \xi_{\ell-1}(\vx,
      \vy, \vq) \nonumber
    \\ & \quad \cdot [^{q_{\ell}-2} \mu_{\ell}(\vx, \vy,
      \vq)]^{q_{\ell}-2}\cdot\tp_a( y_{\ell+1}  \cdot
         \xi_{\ell+1}(\vx, \vy, \vq) \cdots  \xi_{n-1}(\vx, \vy,
         \vq))\label{eq:46}
    \\  \tp_a(\Tt_k(\w)) &=  \Tt_k(y_0) \cdot \xi_0(\vx, \vy,
         \vq)  \cdot \xi_{1}(\vx, \vy, \vq) \cdots  \xi_{\ell-1}(\vx,
         \vy, \vq) \nonumber
    \\ & \quad\cdot [^{q_{\ell}-1} \mu_{\ell}(\vx, \vy,
         \vq)]^{q_{\ell}-1}\cdot
         \tp_a( y_{\ell+1}  \cdot
         \xi_{\ell+1}(\vx, \vy, \vq) \cdots  \xi_{n-1}(\vx, \vy,
         \vq)). \label{eq:59}
  \end{align}
  Moreover, using again Lemma \ref{5.7comp}, we obtain
  \begin{align}
    \w(i,a) &= \tp_a(\Tt_i(\w)) = \tp_a(\Tt_i(\mu_n(\vx, \vy,
              \vq)))\nonumber
    \\ & =  \tp_a(\Tt_i (y_\ell)) \cdot
         \xi_\ell(\vx, \vy, \vq)  \cdot \xi_{\ell+1}(\vx, \vy, \vq)
         \cdots \xi_{n-1}(\vx, \vy, \vq)\nonumber
    \\ & = \Tt_i (y_\ell)[^ {q_\ell-1}\mu_\ell(\vx, \vy,
         \vq)]^{q_\ell-1}\nonumber
         \tp_a(
         y_{\ell+1}\xi_{\ell+1}(\vx, \vy, \vq)
         \cdots \xi_{n-1}(\vx, \vy, \vq)) \nonumber
    \\ & =  \Tt_i (y_\ell)[^ {q_\ell-1}
         x_\ell[^{q_{\ell-1}}x_{\ell-1} [^{q_{\ell-2}}\cdots
         [^{q_0}x_0y_0]^{q_0}
         \cdots]^{q_{\ell-2}}y_{\ell-1}]^{q_{\ell-1}}y_\ell]^{q_\ell-1}\nonumber 
    \\ & \quad \cdot\tp_a(
         y_{\ell+1}\xi_{\ell+1}(\vx, \vy, \vq)
         \cdots \xi_{n-1}(\vx, \vy, \vq))
         \label{eq:47}
  \end{align}
  Since $b_k$ is the principal marker of $\w(i,a)$, we know that the
  following inclusion holds:
  $$c_A(y_0 y_1 \cdots y_\ell \cdot \tp_a( y_{\ell+1}  \cdot
  \xi_{\ell+1}(\vx, \vy, \vq) \cdots  \xi_{n-1}(\vx, \vy, \vq))
  ) \subseteq c_A(\Tt_i(y_\ell)x_\ell\cdots x_0 b_k).$$
  Also,  by definition of $\mu_\ell(\vx, \vy, \vq)$, we have an
  inclusion
  $$c_A(\Tt_i(y_\ell)x_\ell\cdots x_0 b_k) \subseteq
  c_A(\mu_\ell(\vx,\vy,\vq)).$$
  Consequently, we obtain
  $$c_A( \tp_a( y_{\ell+1}  \cdot
  \xi_{\ell+1}(\vx, \vy, \vq) \cdots  \xi_{n-1}(\vx, \vy, \vq))) \subseteq
  c_A(\mu_\ell(\vx,\vy,\vq)).$$
  Observing that
  \begin{align}
    \cum{\eta([^{q_\ell-2} \mu_\ell(\vx,\vy,\vq)]^{q_\ell-2})} = c(\eta(\tp_a( y_{\ell+1}  \cdot
         \xi_{\ell+1}(\vx, \vy, \vq) \cdots  \xi_{n-1}(\vx, \vy,
         \vq)))),\label{eq:12}
  \end{align}
  we end up with the desired relations, which are valid in $\DRH$:
  \begin{align*}
    \eta(\Tt_k(\w(i,a)))
    & \just = {\eqref{eq:11},\eqref{eq:46}} \eta(\Tt_k(y_0) \cdot
      \xi_0(\vx, \vy, \vq)
    \cdot   \xi_{1}(\vx, \vy, \vq) \cdots
      \xi_{\ell-1}(\vx, \vy,\vq) )
    \\ & \kern15pt\quad \cdot \eta([^{q_{\ell}-2} \mu_{\ell}(\vx, \vy,
         \vq)]^{q_{\ell}-2} )
    \\ & \kern15pt\quad\cdot \eta(\tp_a( y_{\ell+1}  \cdot
         \xi_{\ell+1}(\vx, \vy, \vq) \cdots  \xi_{n-1}(\vx, \vy, \vq)))
    \\ & \kern10pt\just {\Req}{\eqref{eq:12}} \eta(\Tt_k(y_0) \cdot
         \xi_0(\vx, \vy, \vq)  \cdot 
         \xi_{1}(\vx, \vy, \vq) \cdots  \xi_{\ell-1}(\vx, \vy, \vq))
    \\ &\kern15pt\quad \cdot \eta(\mu_\ell(\vx, \vy,
         \vq))^{\omega+q_\ell-1}
    \\ &\kern15pt\quad\cdot \eta(\tp_a( y_{\ell+1}  \cdot 
         \xi_{\ell+1}(\vx, \vy, \vq) \cdots  \xi_{n-1}(\vx, \vy, \vq)))
    \\ & \kern10pt\just = {\eqref{eq:59}} \eta(\w(k,a)) = w(k,a). 
  \end{align*}
  We finally observe that we actually proved an equality in $\pseudo
  AS$ rather than a
  relation modulo $\DRH$, except in the last situation. But that
  scenario only occurs when $w(i,a)$ is regular modulo $\DRH$.
  Indeed, since $b_k \in c(y_0)$ is the principal marker of $w(i,a)$,
  from the equality \eqref{eq:47}, we may deduce that $\cum{w(i,a)} =
  c(w(i,a))$, which by Proposition~\ref{c:8} implies that $\rho_{\DRH}(w(i,a))$
  is regular.
\end{proof}

For a well-parenthesized word $x$ over $A\times \nn$, we consider the following property:
\begin{equation}
  \tag{$H(x)$}
  \label{hx}
  \forall a,b \in A, \quad \forall i \in \nn, \quad a_i,b_i \in c(x)
  \implies a = b
\end{equation}
The proof of the next result may be easily adapted from the proof of~\cite[Lemma 5.9]{palavra}.
\begin{lem}\label{5.9}
  Let $x \in \dyck{A \times \nn} \setminus \{\varepsilon\}$ satisfy
  \eqref{hx} and suppose
  that $a_i$ is a marker of $x$. Then the equality $\eta(x) =
  \eta(\tp_a(x) \cdot a_i \cdot \Tt_i(x))$ holds.
\end{lem}
\begin{cor}\label{5.11}
  Let $w$ be a $\overline\kappa$-term. Let $i \in \nn$ and $a \in A \uplus\{\#\}$,
  and let $b_k$ be the principal marker of $\overline{w}(i,a)$. Suppose
  that $\lbf{w(i,a)} = (w_\ell, m, w_r)$. Then, $m = b$ and~$\DRH$
  satisfies $w_\ell = w(i,b)$, and $ w_r \Req w(k,a)$.
  Moreover, if $\rho_\DRH(w(i,a))$ is not regular, then
  $\lbf{w(i,a)} = (w(i,b), b, w(k,a))$.
\end{cor}
\begin{proof}
  As $b_k$ is the principal marker of $\w(i,a)$, we can write $\w(i,a) =
  x b_k y$, where $c_A(y) \subseteq c_A(xb_k)$ and $b \notin
  c_A(x)$. Since \eqref{hx} holds,  Lemma \ref{5.9} yields
  $$\eta(\w(i,a)) = \eta(\tp_b(\w(i,a)) \cdot b_k \cdot \Tt_k(\w(i,a))) =
  \eta(\tp_b(\w(i,a))) \cdot b \cdot \eta(\Tt_k(\w(i,a))).$$
  Furthermore, since $b \notin c_A(x)$, we also have $c_A(\tp_b(\w(i,a)))
  = c_A(x)$ and consequently, the left basic factorization of $w(i,a)$
  is precisely $$(\eta(\tp_b(\w(i,a))), b, \eta(\Tt_k(\w(i,a)))).$$ In
  particular, we have $m = b$ and, by Lemma \ref{5.8}, the
  pseudovariety $\DRH$ satisfies $w_\ell = w(i,b)$ and $ w_r \Req
  w(k,a)$, with an equality in $\S$ in the latter relation when
  $w(i,a)$ is not regular modulo $\DRH$.
\end{proof}
\section{$\DRH$-graphs and their computation}\label{section6}

We begin this section with the definition of a $\DRH$-graph. Through
these structures, we are able to decide whether two $\kappa$-words are
$\Req$-equivalent over $\DRH$. If we further assume that the word
problem is decidable in $\pseudok AH$, then the word
problem is decidable in $\pseudok A{DRH}$ as well.

\begin{deff}
  Let $w$ be a $\overline\kappa$-term.  The \emph{$\DRH$-graph of $w$}
  is the finite $\DRH$-automaton
  $$\iG(w) = \langle V(w), \to,\tq(0, \#), \{\varepsilon\},
  \lambda_{\h}, \lambda \rangle,$$
  defined as follows.
  The set of states is
  $$V(w) = \{\tq(i, a)\colon 0 \le i < \card \w, \: a \in c_A(\w)
  \text{ and } w(i,a) \neq I\}\uplus \{\varepsilon\}.$$
  Let $\tq(i,a) \in V(w)\setminus \{\varepsilon\}$ and $b_k$
  be the principal marker of $\w(i,a)$. The transitions of $\tq(i,a)$
  are $\tq(i,a).0 = \tq(i, b)$ and $\tq(i,a).1 = \tq(k, a)$.
  The labels are $\lambda_{\h}(\tq(i,a)) =
  \rho_\h(\reg{w(i,b)})$ and $\lambda(\tq(i,a)) = b$.
  If a state $\tq(i,a)$ is not reached from the root $\tq(0, \#)$,
  then we discard it from $V(w)$.
\end{deff}

The following result suggests that the construction of $\iG(w)$ might
be a starting point to solve the $\kappa$-word problem over $\DRH$
algorithmically.
\begin{prop}\label{p:9}
  For every $\overline\kappa$-term $w$, $\iG(w)$ is a $\DRH$-automaton
  equivalent to $\iT(w(0,\#))$.
\end{prop}
\begin{proof}
  Let
  \begin{align*}
    \iT(w(0,\#)) &= \langle V, \to_{\iT}, \tq, F, \lambda_{\iT,\h},
                   \lambda_{\iT} \rangle,
    \\ \iG(w) &= \langle V(w), \to_{\iG}, \tq(0,
  \#), \{\varepsilon\}, \lambda_{\iG,\h}, \lambda_{\iG} \rangle.
  \end{align*}
  We
  first claim that, for every $\alpha \in \Sigma^*$, we have
  \begin{equation}
    \label{eq:14}
    \tq.\alpha = \tq(i,a) \implies \iT(w(0,\#))_{\tq.\alpha} = \iT(w(i,a)).
  \end{equation}
  To prove this, we argue by induction on $\card \alpha$. If $\card
  \alpha = 0$, then the result holds trivially. Let $\alpha \in \Sigma^*$ be such that
  $\card \alpha \ge 1$ and suppose that the result holds for every
  other shorter word~$\alpha$. We can write $\alpha = \beta \gamma$,
  with $\gamma \in \{0,1\}$. Let
  $\tq.\beta = \tq(i,a)$. By induction hypothesis, it follows that
  $\iT(w(0, \#))_{\tq.\beta} = \iT(w(i,a))$. Let $b_k$ be the principal
  marker of $\w(i,a)$. By definition of $\iG(w)$, we have
  $\tq(0, \#).\beta0 = \tq(i,b)$ and $\tq(0, \#).\beta1 = \tq(k,a)$.
  On the other hand, Lemma \ref{arv} gives that if $\lbf{w(i,a)} =
  (w_\ell, b, w_r)$, then
  $\iT(w(i,a)) = (\iT(w_\ell), \reg{w_\ell} \mid b, \iT(w_r))$,
  which in turn, by Corollary \ref{5.11}, is equivalent to
  \begin{align}
    \iT(w(i,a)) = (\iT(w(i,b)), \reg{w(i,b)} \mid b,
    \iT(w(k,a))).\label{eq:60}
  \end{align}
  In particular,
  we conclude that
  $\iT(w(0,\#))_{\tq.\beta0} = \iT(w(i,b)) \text{ and } \iT(w)_{\tq.\beta1}
  = \iT(w(k,a))$.
  It is now enough to notice that, for each pair $(i,a) \in
  {[}0,\card \w{[}\times c_A(\w)$, the labels of the node $\tq(i,a)$ of
  $\iG(w)$ and the labels of the root of $\iT(w(i,a))$ coincide. In
  fact, if $b_k$ is the principal marker of $\w(i,a)$, then the
  construction of $\iG(w)$ yields the equalities
  $\lambda_{\iG}(\tq(i,a)) = b $ and $\lambda_{\iG,\h}(\tq(i,a)) =
  \rho_{\h}(\reg{w(i,b)})$, which, by \eqref{eq:60}, are precisely the
  labels of the root of~$\iT(w(i,a))$.
\end{proof}

Imagine we are given a $\kappa$-word and let $w = a^{\omega
  + q}$ be one of its representations as a $\overline\kappa$-term, with $q$ ``very
big''. Then, we have $\w = 0_0[^qa_1]^q\#_2$ and so, $\card \w = 3$.
Conceptually speaking,
such a $\kappa$-word involves a ``large'' number
of implicit operations of $\kappa$ but the length of its
representation $\w$ in $\dyck{A\times \nn}$ is just $3$.
Therefore, allowing any representation of $\kappa$-words,
we would not be able to get meaningful results for
the efficiency of the forthcoming algorithms.
Thus, it is reasonable to require that all
$\kappa$-words are presented as $\kappa$-terms.
We make that assumption from now on.

Consider a $\kappa$-term $w$.
We may assume that $w$ is given by a tree.
For instance, if $w = ((((b^{\omega-1})\cdot a)\cdot c) \cdot (((a\cdot
b)\cdot (a^{\omega-1}))^{\omega-1}))$,
then the tree representing $w$  is depicted
in Figure~\ref{fig:12}.
\begin{figure}[h]
  \centering
  \begin{tikzpicture}[sibling distance=2cm, level 2/.style={sibling
      distance =1.5cm}, level distance = 0.7cm]
    \tikzstyle{bullet}=[inner sep = -50]
    \coordinate
    node[bullet]{\footnotesize$\bullet$}
    child{ node[bullet]{\footnotesize$\bullet$}
      child {node[bullet]{\footnotesize$\bullet$}
        child {node[bullet]{\footnotesize$\bullet$}
          child {node {$b$}}
        }
        child {node {$a$}}
      }
      child {node {$c$}}}
    child { node[bullet]{\footnotesize$\bullet$}
      child {node[bullet]{\footnotesize$\bullet$}
        child {node[bullet]{\footnotesize$\bullet$}
          child {node {$a$}}
          child {node {$b$}}}
        child {node[bullet]{\footnotesize$\bullet$}
          child{ node {$a$}}}}};
  \end{tikzpicture}
  \caption{The tree representing $((((b^{\omega-1})\cdot a)\cdot c) \cdot (((a\cdot
    b)\cdot (a^{\omega-1}))^{\omega-1}))$.}
  \label{fig:12}
\end{figure}
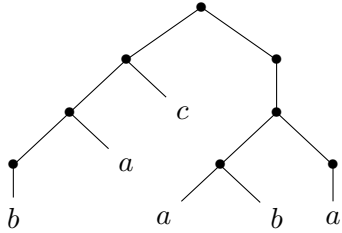
Since from such a tree representation we may compute $\w$ in linear
time, we assume that we are already given $\w$.
If the tree representing $w$ has $n$ nodes then,
following~\cite{palavra}, we say that the \emph{length} of $w$ is
$\card w = n + 1$. 
It is clear that $O(\card w) = O(\card \w)$.
To actually compute the $\DRH$-graph $\iG(w)$ we essentially need to
compute the principal marker of the  words $\w(i, a)$ as well
as the regular parts of $w(i,a)$.
Almeida and Zeitoun~\cite{palavra}
exhibited an algorithm to compute the first occurrences of each letter
of a well-parenthesized word~$x$.
Given a word $x$, ${\sf first}(x)$ consists of a list of the  first
occurrences of each letter in $x$.
In particular, this computes the principal marker of
$x$: it is the last entry of the outputted list.
Moreover, if $b_k$ is the principal marker of $x$, then the
penultimate entry of the list is the principal marker of $\tp_b(x)$,
and so on.
Hence, this is enough to almost compute $\iG(w)$.
More precisely, the knowledge of ${\sf first}(\w(i,a))$, for every
pair $(i,a)$, allows us to compute the reduct~$\iG_\R(w) = \langle
V(w), \to,\tq(0, \#), \{\varepsilon\}, \lambda \rangle$ in time
$O(\card w \card{c(w)})$.

\begin{lem}
  [{\cite[Lemma 5.15]{palavra}}]
  \label{l:1}
  Let $w$ be a $\overline\kappa$-term. Then, one may compute in time
  $O(\card w\card{c(w)} )$ a table giving, for each $i$ such
  there exists $a_i \in c(\w) \cap A \times \nn$, the word ${\sf
    first}(\w(i, \#))$.
\end{lem}

It remains to find the labels of the states under $\lambda_\h$.
For that purpose, we observe that the regular part of a pseudoword $u$ depends
deeply on the content of the factors of the form~$\lbf[k] u$,
which we may compute using Lemma \ref{5.9}; and of the cumulative
content of~$u$.
Also, it follows from Lemma \ref{l:2} and from Proposition \ref{p:9} that
the cumulative content of any pseudoword of the form $w(i,a)$ is
completely determined by the reduct $\iG_\R(w)$.
Thus, we may start by computing the cumulative content of $w(i,a)$ and
then compare it with the content of~$\lbf[k] {w(i,a)}$,
for increasing values of $k$. When we achieve an
equality, we know what is the regular part of $w(i,a)$. Algorithm~\ref{a:1} does that 
job. We assume that we already have the table described in
Lemma~\ref{l:1}, so that, computing $c(w(i,a))$ and 
the principal marker 
of $\w(i,a)$ takes {$O(1)$-time}.
Further, we may assume that we are given $\iG_\R(w)$, since we already
explained how to get it from the table of Lemma~\ref{l:1}
in $O(\card w \card{c(w)})$-time.

\begin{algorithm}[htpb]
  \begin{algorithmic}[1]
    \Require A $\overline\kappa$-term $w$ and  $(i, a) \in [0,\card \w{[}\times
    c_A(\w)$ (with $\w(i,a) \neq \varepsilon$)
    \Ensure $\reg{w(i,a)} = I$, if $\cum{w(i,a)} = \emptyset$ or $k$
    such that $\reg{w(i,a)} = w(k,a)$, otherwise
    \State $L \gets \{\}$, $j \gets i$
    \While {$j \notin L$ and $\w(j,a) \neq \varepsilon$}\label{op:3}
    \State {$j \gets \pi_\nn(\text{principal marker of
          $\w(j,a)$})$} \Comment{\parbox[t]{.35\textwidth}{So that, if $\tq(j,a).1 \neq
        \varepsilon$, then $\tq(j,a) \gets \tq(j,a).1$}}
    \State {$L \gets L \cup \{j\}$}
    \EndWhile
    \If {$\w(j,a) = \varepsilon$}
    \State{\bf return $I$}\label{op:39}
    \Else\label{op:41}
    \State $C \gets c(w(j,a))$ \Comment{The
      set $C$ is the cumulative content of $w(i,a)$} \label{op:1}
    \State $k \gets i$ \label{op:8}
    \While {$c_A(\w(k, a)) \neq C$} \label{op:9}
    \State $k \gets \pi_\nn(\text{principal marker of $\w(k, a)$})$
    \EndWhile\label{op:10}
    \State \textbf{return} $k$\label{op:11}
    \EndIf\label{op:42}
  \end{algorithmic}
  \caption{}
  \label{a:1}
\end{algorithm}

\begin{lem}\label{l:12}
  Algorithm \ref{a:1} returns $I$ if and only if $\cum{w(i,a)} =
  \emptyset$.
  Otherwise, the value $k$ outputted is such that $\reg
  {w(i,a)} = w(k,a)$.
  Moreover, the algorithm runs in linear time, provided we have the
  knowledge of ${\sf first}(w(i,a))$.
\end{lem}
\begin{proof}
  By Property \ref{a3} of a $\DRH$-automaton, and since there is only
  a finite number of possible states in $\iG_\R(w)_{\tq(i,a)}$, either there
  exists $k \ge 0$ such that $\tq(i,a).1^k = \varepsilon$, or there
  exists $\ell > k \ge 0$ such that $\tq(i,a).1^k = \tq(i, a).1^\ell$.
  Therefore, the cycle {\bf while} in line~\ref{op:3} does not run forever.
  If the occurring situation is the former, then
  $\cum{\iG(w)_{\tq(i,a)}} = \emptyset$.
  On the other hand, by Proposition \ref{p:9}, we have
  $\iG(w)_{\tq(i,a)}\sim \iT(w(i,a))$ which in turn, by Theorem
  \ref{3.21}, implies
  $\pi(\iG(w)_{\tq(i,a)}) \Req w(i,a)$ modulo $\DRH$. Also, Lemma \ref{l:2}
  yields $ \cum{w(i,a)}= \cum{\iG(w)_{\tq(i,a)}} = \emptyset$,
  and therefore, $\reg{w(i,a)} = I$.
  This is the case where the symbol $I$ is returned in line \ref{op:39}.

  Now, suppose that $\ell > k \ge 0$ are such that $\tq(i,a).1^k =
  \tq(i, a).1^\ell$.
  Then, the cycle {\bf while} is exited because an index $j$ is
  repeated.
  By Property \ref{a4}, we have a chain
  of inclusions:
  $\lambda(\iG(w)_{\tq(i,a).1^k})\supseteq
  \lambda(\iG(w)_{\tq(i,a).1^{k+1}})\supseteq \cdots
  \supseteq \lambda(\iG(w)_{\tq(i,a).1^{\ell}})$.
  As $\tq(i,a).1^k =  \tq(i, a).1^\ell$, these inclusions are
  actually equalities, implying that $k$ is greater  than or equal to $\regi{\iG(w)_{\tq(i,a)}}$.
  Combining again Proposition \ref{p:9}, Theorem \ref{3.21} and Lemma~\ref{l:2}, we may deduce that
  $\cum{w(i,a)} = \cum{\iG(w)_{\tq(i,a)}} =
  \lambda(\iG(w)_{\tq(i,a).1^k})$,
  where the last member is precisely $c(w(j,a))$ provided that
  $\tq(i,a).1^k = \tq(j,a)$.
  Therefore, in line \ref{op:1} we assign to $C$ the cumulative content
  of $w(i,a)$.
  Until now, since we are assuming that we are given all the
  information about $\iG_\R(w)$, we only spend time $O(\card w)$,
  because that is the number of possible values of $j$ that may appear in
  line \ref{op:3}.

  Let us prove that, if we get to line \ref{op:40}, then the value $k$
  outputted in line \ref{op:11} is such that 
  $ \reg{w(i,a)} = w(k,a)$.
  We write
  $$w(i,a) = \lbf[1]{w(i,a)}\cdots \lbf[m]{w(i,a)}w_m',$$
  for every $m \ge 1$ (notice that $\lbf[m]{w(i,a)}$ is defined for all
  $m \ge 1$ because we are assuming that $\cum{w(i,a)} \neq
  \emptyset$).
  Then, the regular part of
  $w(i,a)$ is given by $w_\ell'$, where $\ell = \min\{m \colon c(w_m')
  = \cum{w(i,a)}\}$. In particular, the projection of $w_m'$ onto
  $\pseudo A{DRH}$ is not regular, for every $m< \ell$.
  Set $(c_0, k_0) = (a, i)$ and, for $m
  \ge 0$, let $(c_{m+1}, k_{m+1})$ be the principal marker of $\w(k_m,
  a)$. By Corollary \ref{5.11}, if
  $w(k_m, a)$ is not regular modulo $\DRH$, then we have $\lbf{w(k_m, a)} =
  (w(k_m, c_{m+1}),c_{m+1},w(k_{m+1},a))$. Therefore, the equality
  $w_m' = w(k_{m},a)$ holds, for every $m \le
  \ell$. Thus, the value $k$ returned in line
  \ref{op:11} is precisely $k_\ell$, implying that $\reg{w(i,a)} =
  w(k,a)$ as intended.

  Since there are only $O(\card\w)$ possible values for $k$ and we are
  assuming that we already know ${\sf first (w(i, \#))}$ for all $i
  \in [0, \card\w{[}$, it follows that lines \ref{op:41}--\ref{op:42} run
  in time $O(\card\w)$.

  Therefore, the overall time complexity of Algorithm \ref{a:1} is $O(\card
  w)$.
\end{proof}

So far, we possess all the needed information for computing
$\iG(w)$.
Putting all the steps together, we obtain the following.

\begin{thm}\label{t:2}
  Given a $\kappa$-term $w$, it is possible to compute the $\DRH$-graph of $w$
  in time~$O(\card w^2 \card{c(w)})$.\qed
\end{thm}

The next question we should answer is how can we decide whether two
$\DRH$-graphs $\iG(u)$ and~$\iG(v)$ represent the same element of
$\DRH$, that is, whether $\iG(u) \sim \iG(v)$.
A possible strategy consists in visiting states in both $\DRH$-graphs,
comparing their labels (in a certain order).
When we find a pair of mismatching labels, we stop, concluding that
$\iG(u)$ and $\iG(v)$ are not equivalent.
Otherwise, we conclude that they are equivalent after visiting all the
states.
More precisely, starting in the roots of $\iG(u)$ and $\iG(v)$, we
mark the current states, say $\tq_u \in V(u)$ and $\tq_v \in V(v)$, as
visited, and then repeat the process relatively to
the pairs of $\DRH$-automata $(\iG(u)_{\tq_u.0}, \iG(v)_{\tq_v.0})$
and $(\iG(u)_{\tq_u.1}, \iG(v)_{\tq_v.1})$.
For a better understanding of the procedure, we sketch it in
Algorithm \ref{a:20}.
\begin{algorithm}[htpb]
  \begin{algorithmic}[1]
    \Require {two $\DRH$-graphs $\iG_i = \langle V_i, \to_i, \tq_i,
      \lambda_{i,\h}, \lambda_i \rangle$ ($i = 1,2$)}
    \Ensure {logical value of ``$\iG_1 \sim \iG_2$''}
    \If {$\tq_1 = \varepsilon$}
    \State {\bf return} logical value of $\tq_2 = \varepsilon$
    \ElsIf {$\tq_1$ or $\tq_2$ is unvisited}
    \State {mark $\tq_1$ and $\tq_2$ as visited}
    \If {$\lambda_{1,\h}(\tq_1) =\lambda_{2,\h}(\tq_2)$ and
      $\lambda_{1}(\tq_1) =\lambda_{2}(\tq_2)$}\label{op:40}
    \State {\bf return} logical value of ``$(\iG_1)_{\tq_1.0} \sim
    (\iG_2)_{\tq_2.0}$ and 
    $(\iG_1)_{\tq_1.1} \sim (\iG_2)_{\tq_2.1}$''
    \Else
    \State {\bf return} {\sf False}
    \EndIf
    \Else
    \State {\bf return} logical value of $(\lambda_{1,\h}(\tq_1),
    \lambda_{1}(\tq_1)) = (\lambda_{2,\h}(\tq_2), \lambda_{2}(\tq_2))$
    \EndIf
  \end{algorithmic}
  \caption{}
  \label{a:20}
\end{algorithm}

\begin{lem}
  Algorithm \ref{a:20} returns the logical value of ``$\iG_1 \sim
  \iG_2$'' for two input $\DRH$-graphs $\iG_1$ and $\iG_2$.
  Moreover, it runs in time $O(p\max\{\card
  {V_1}, \card {V_2}\})$, where $p$ is such that the word problem
  modulo $\h$ for any pair of labels $\lambda_{1,\h}(\tv_1)$ and
  $\lambda_{2,\h}(\tv_2)$ (with $\tv_1 \in V_1$ and $\tv_2 \in V_2$) may be
  solved in time $O(p)$.
\end{lem}
\begin{proof}
  The correctness follows straightforwardly from the definition of the
  relation $\sim$.
  On the other hand, it runs in time $O(p\max\{\card{V_1}, \card{V_2}\})$,
  since each call of the algorithm takes time $O(p)$ (line \ref{op:40}) and
  each pair of states of the form ($\tq_1.\alpha$, $\tq_2.\alpha$) is
  visited exactly once.
\end{proof}

Given $\overline\kappa$-terms $u$ and $v$, we use $p(u,v)$ to denote a
function depending on some parameters associated with $u$ and $v$
(that may be, for instance, $\card u$, $\card v$ or $c(u)$, $c(v)$)
and such that, the time for solving the word problem over $\h$ for any
pair of factors of the form $u(i,a)$ and $v(j,b)$ is in
$O(p(u,v))$.
Observe that the time to transform an expression of the form
$u(i,a)$ into a $\overline\kappa$-term should be taken into
account.
Furthermore, such a function $p(u,v)$ is not unique, but the results are valid
for any such function.
Then, summing up the time complexities of all the intermediate steps considered
above, we have just proved the following result.

\begin{thm}\label{t:7}
  Let $\h$ be a  pseudovariety of groups with decidable
    $\kappa$-word problem,
  and let $u$ and $v$ be $\kappa$-terms.
  Then, the equality of the pseudowords represented by
  $u$ and $v$ over $\DRH$ can be tested in time
  $O((p(u,v) + m) m \card A)$, where $m = \max\{\card u, \card v\}$. \qed
\end{thm}
Observe that, in general, the complexity of an algorithm for solving
the $\kappa$-word problem over~$\h$ should depend on the length of the
intervening $\overline\kappa$-terms.
It is not hard to see that the length of the factors $w(i,a)$ grows
quadratically on $\card w$ (we prove it below in Corollary \ref{c:15}).
Hence, it is expected that, at least in most of the cases, $m$
belongs to $O(p(u,v))$.
Consequently, the overall time complexity stated in Theorem \ref{t:7} becomes
$O(p(u,v) m \card A)$.
Since we are doing the same approach as in \cite{palavra}, this result is
somehow the expected one.
Roughly speaking, this may be interpreted as the time complexity of solving
the word problem in~$\R$,
together with a word problem in $\h$ for each state, that is, for each
$\DRH$-factor of the involved pseudowords (recall Lemmas~\ref{sec:16} and~\ref{l:3}).

Just as a complement, we mention that
another possible approach would be to transform the $\DRH$-graph
$\iG(w)$ in an automaton in the classical sense, say $\iG'(w)$, recognizing the
language $\iL(w)$.
That is easily done (time linear on the number of states), by moving the
labels of a state to the arrows leaving it. More
precisely, the automaton $\iG'(w)$ shares the set of states with
$\iG(w)$ and
each non terminal state $\tq(i,a)$ has two transitions:
\begin{align*}
  \tq(i,a).(0,\lambda_\h(\tq(i,a)),\lambda(\tq(i,a)))
  & = \tq(i,1).0
  \\ \tq(i,a).(1,I, \lambda(\tq(i,a))) & = \tq(i,a).1.
\end{align*}
Then, we could use the results in the literature in order to
minimize the automaton, obtaining a unique automaton representing each
$\Req$-class of the semigroup $(\pseudo A{DRH})^I$.
The unique issue in that approach is that the algorithms are usually
prepared to deal with alphabets whose members may be compared in
constant time.
Hence, we should previously prepare the input automaton by renaming
the subset of the alphabet
$\Sigma \times (\pseudo AH)^I \times A$, in which the labels of
transitions are being considered.
Let $p(u,v)$ and $m$ have the same meaning has in Theorem \ref{t:7}.
Since, a priori, we do not possess any information about the possible values for
$\lambda_\h$, that would take $O(p(u,v) (m \card A)^2)$-time (each time we rename an
element in $(\pseudo AH)^I$ we should first verify whether we already
encountered another element with the same value over $\h$).
Thereafter, we could use the linear time algorithm presented in
\cite{dfamin}, which works for this kind of automaton.
Thus, a rough upper bound for the complexity spent using this method is
$O(p(u,v)m^2\card A^2)$, which although a bit worse, is still polynomial.

The following result gives us a family of pseudovarieties of the
form $\DRH$ with decidable $\kappa$-word problem.
It is a consequence of the fact that the free group is residually in~$\G_p$.
\begin{cor}
  Let $p$ be a prime number.
  If $\h \supseteq \G_p$ is a pseudovariety of groups, then the
  pseudovariety $\DRH$ has
  decidable $\kappa$-word problem.\qed
\end{cor}
\section{An application: solving the word problem over ${\sf DRG}$}\label{section7}

Let us illustrate the previous results by considering the particular
case of the pseudovariety~${\sf DRG}$.
By Theorem \ref{t:7}, the time complexity of our procedure for
testing identities of $\kappa$-terms modulo ${\sf DRG}$ depends on a
certain parameter $p(\_,\_)$.
In order to discover that parameter, we should first analyze the
(length of the)
projection onto $\pseudok AG = \FG A$ of the elements of the form $w(i,a)$,
where $w$ is a $\kappa$-term.

Consider the alphabets $B_1 = (A \times \nn) \uplus
\{[^{-1},{]}^{-1}\}$ and
$B_2 = (A \times \nn) \uplus \{[^{-1},[^{-2},{]}^{-1},{]}^{-2}\}$.
Let $x$ be a well-parenthesized word over $B_2$. The \emph{expansion}
of $x$ is the well-parenthesized
word ${\sf exp}(x)$
obtained by successively applying the rewriting
rule $[^{-2}y]^{-2} \to [^{-1}y]^{-1}[^{-1}y]^{-1}$, whenever $y$ is a
well-parenthesized word.
It is clear that $\om x $ and $\om{{\sf exp} (x)}$ represent the same
$\kappa$-word and that $x$ is a 
well-parenthesized word over $B_1$.
Further, we have the following.
\begin{lem}\label{l:9}
  Let $x$ be a nonempty well-parenthesized word over $B_1$ and $i \in
  c_\nn(x)$.
  Then, $\Tt_i(x)$ is a well-parenthesized word over $B_2$ and
  $\card{{\sf exp}(\Tt_i(x))} \le \frac{1}{2} (\card x^2 + 2\card x -3)$.
  Moreover, this upper bound is tight for all odd values of $\card x$.
\end{lem}
\begin{proof}
  The fact that $\Tt_i(x)$ is a well-parenthesized word over $B_2$
  follows immediately from the definition of $\Tt_i$.
  To prove the inequality, we proceed by induction on $\card x$.
  If $x = a_i$, then $\Tt_i(x)$ is the empty word and so, the result holds.
  Let $x$ be a well-parenthesized word over $B_1$ such that $\card x =
  n$.
  The inequality holds clearly, unless $x$ is of the form
  $x = [^{-1}y{]}^{-1} z$, with $y$ and $z$ well-parenthesized words
  over $B_1$, $y$ nonempty and $i \in c_\nn(y)$.
  In that case, we have $\Tt_i(x) = \Tt_i(y) [^{-2}y{]}^{-2} z $.
  Using induction hypothesis on $y$, one may deduce that
  $\card{{\sf exp}(\Tt_i(x))} \le \frac{1}{2} (\card x^2 + 2\card x -3)$.
  Finally, let $\vx = (a_1,\varepsilon,\varepsilon,\ldots)$, $\vy =
  (\varepsilon,\varepsilon,\ldots)$, $\vq = (-1,-1,\ldots)$, and
  $u_{2n+1} =  \mu_n(\vx,\vy,\vq)$
  (recall the notation used in Lemma \ref{5.7comp}).
  Then, $u_{2n+1}$ is a well-parenthesized word over $B_1$ of
  length $2n+1$.
  Moreover, using Lemma \ref{5.7comp}, we may compute
  \begin{align*}
    \card{{\sf exp}(\Tt_1(u_{2n+1}))} & = \card{ {\sf exp}(\Tt_1 (\mu_0(\vx, \vy, \vq)) \cdot
    \xi_0(\vx, \vy, \vq)  \cdot \xi_{1}(\vx, \vy, \vq) \cdots
                                        \xi_{n-1}(\vx, \vy, \vq))}
    \\ & = \card{{\sf exp} \left([^{-2}\mu_0(\vx, \vy, \vq)]^{-2}
         \cdots
         [^{-2}\mu_{n-1}(\vx, \vy, \vq)]^{-2}\right)}
    \\ & = \sum_{k = 0}^{n-1}2(\card{\mu_k(\vx, \vy, \vq)} + 2)
    \\ & = 2n^2 + 4n \quad \text{because $\card{\mu_k(\vx, \vy, \vq)}
         = 2k+1$}
    \\ & = \frac 12 (\card{u_{2n+1}}^2 + 2\card{u_{2n+1}} - 3)
  \end{align*}
  and the result follows.
\end{proof}

Also, as a straightforward consequence of the definition of $\tp_a$,
the following holds.
\begin{lem}\label{l:11}
  Let $x$ be a nonempty well-parenthesized word over $B_1$ and $a \in
  A$.
  Then, $\tp_a(x)$ is also a well-parenthesized word over $B_1$ and
  $\card{{\sf exp}(\tp_a(x))} = \card{\tp_a(x)}\le \card x$.\qed
\end{lem}

Given a well-parenthesized word $x$ over $B_2$, we define the
\emph{linearization over $A$} of $x$ to be the word ${\sf lin}(x)$
over the alphabet 
$A\uplus A^{-1}$ obtained by applying the rewriting rules
$[^{-1}a_i]^{-1} \to a^{-1}$, $[^{-1}yz{]}^{-1} \to
[^{-1}z]^{-1}[^{-1}y]^{-1}$ and $[^{-2}y]^{-2} \to
[^{-1}y]^{-1}[^{-1}y]^{-1}$ to $x$ (with $a_i \in c(x)$ and
$y$, $z$ well-parenthesized words).
It is easy to see that ${\sf lin}(x) = {\sf lin}({\sf exp}(x))$ and
that if $x$ is a well-parenthesized word over $B_1$, then $O(\card{{\sf
    lin}(x)}) = O (\card x)$.
Consequently, we have the next result.

\begin{cor}\label{c:15}
  Let $w$ be an $\kappa$-term
  and $(i,a) \in [0, \card{\overline w}{[} \times c_A(\w)$.
  Then, $\card{{\sf lin}(\w(i,a))}$ belongs to $O(\card w^2)$.\qed
\end{cor}

Now, we wish to compute ${\sf lin} (x)$, for a given well-parenthesized word
over $B_2$.
Recall the tree representation of $\kappa$-terms exemplified in Figure
\ref{fig:12}.
We may recover, in linear time, such a tree
representation for $\om x$, for a well-parenthesized word $x$ over
$B_1$.
Furthermore, if we are given a well-parenthesized word over $B_2$, we
may compute, also in linear time, a tree representation for
$\om{{\sf exp}(x)}$.
That amounts to, whenever we have a factor of the form $[^{-2}y]^{-2}$
in $x$, to include twice a subtree representing $[^{-1}y]^{-1}$.

On the other hand, since solving the word problem in $\FG A$ (for words
written over the alphabet $A \cup A^{-1}$) is a linear issue in the
size of the input, by Corollary \ref{c:15}, we may take $p(u,v) =
\max\{\card u^2, \card v^2\}$. Thus, we have proved the following.

\begin{prop}
  The $\kappa$-word problem over ${\sf DRG}$ is decidable in
  $O( m^3\card A )$-time, where 
  $m$ is the maximum length of the inputs.\qed
\end{prop}

\paragraph{\bf Acknowledgments.}  This work is part of the author's
Ph.D. thesis, written under supervision of Professor Jorge Almeida, to
whom the author is deeply grateful.
The work was partially supported by CMUP
(UID/MAT/ 00144/2013), which is funded by FCT (Portugal) with national
(MEC) and European structural funds through the programs FEDER, under
the partnership agreement PT2020, and by FCT doctoral scholarship
(SFRH/BD/ 75977/2011), with national (MEC) and European
structural funds through the program POCH.
\def\cprime{$'$}
\providecommand{\bysame}{\leavevmode\hbox to3em{\hrulefill}\thinspace}
\providecommand{\MR}{\relax\ifhmode\unskip\space\fi MR }
\providecommand{\MRhref}[2]{%
  \href{http://www.ams.org/mathscinet-getitem?mr=#1}{#2}
}
\providecommand{\href}[2]{#2}


\begin{thebibliography}{10}

\bibitem{livro}
J.~Almeida, \emph{Finite semigroups and universal algebra}, World Scientific
  Publishing Co. Inc., River Edge, NJ, 1994, Translated from the 1992
  Portuguese original and revised by the author.

\bibitem{profinite}
\bysame, \emph{Profinite semigroups and applications}, Structural theory of
  automata, semigroups, and universal algebra, NATO Sci. Ser. II Math. Phys.
  Chem., vol. 207, Springer, Dordrecht, 2005, pp.~1--45.

\bibitem{MR1750489}
J.~Almeida and B.~Steinberg, \emph{Syntactic and global semigroup theory: a
  synthesis approach}, Algorithmic problems in groups and semigroups
  ({L}incoln, {NE}, 1998), Trends Math., Birkh\"auser Boston, Boston, MA, 2000,
  pp.~1--23.

\bibitem{drh}
J.~Almeida and P.~Weil, \emph{Free profinite {$\mathcal R$}-trivial monoids},
  Internat. J. Algebra Comput. \textbf{7} (1997), no.~5, 625--671.

\bibitem{palavra}
J.~Almeida and M.~Zeitoun, \emph{An automata-theoretic approach to the word
  problem for {$\omega$}-terms over~{$\sf{R}$}}, Theoret. Comput. Sci.
  \textbf{370} (2007), no.~1-3, 131--169.

\bibitem{dfamin}
\bysame, \emph{Description and analysis of a bottom-up {DFA} minimization
  algorithm}, Inform. Process. Lett. \textbf{107} (2008), no.~2, 52--59.

\bibitem{phd}
C.~Borlido, \emph{The word problem and some reducibility properties for
  pseudovarieties of the form $\sf{DRH}$}, Ph.D. thesis, University of Porto,
  2015.

\bibitem{MR0530383}
S.~Eilenberg, \emph{Automata, languages, and machines. {V}ol. {B}}, Academic
  Press, New York-London, 1976.

\bibitem{MR2883026}
A.~Moura, \emph{The word problem for {$\omega$}-terms over~{$\sf{DA}$}},
  Theoret. Comput. Sci. \textbf{412} (2011), no.~46, 6556--6569.

\bibitem{MR2567276}
J.~Sakarovitch, \emph{Elements of automata theory}, Cambridge University Press,
  Cambridge, 2009, Translated from the 2003 French original by Reuben Thomas.
  \MR{2567276 (2011g:68003)}

\bibitem{MR0444824}
M.~P. Sch{\"u}tzenberger, \emph{Sur le produit de concat\'enation non ambigu},
  Semigroup Forum \textbf{13} (1976/77), no.~1, 47--75.

\bibitem{topologia}
S.~Willard, \emph{General topology}, Dover Publications, Inc., Mineola, NY,
  2004, Reprint of the 1970 original [Addison-Wesley, Reading, MA].
\end{thebibliography}
\end{document}